%% file: Chaos_in_the_mixed_even-spin_models_Final_version_.tex
\documentclass[12pt]{article}

\usepackage{mathrsfs}

\usepackage{amsmath,amsthm,amssymb}
\usepackage{graphicx}
\usepackage{vector}

\font\tencmmib=cmmib10 \skewchar\tencmmib '60
\newfam\cmmibfam
\textfont\cmmibfam=\tencmmib

\def\lessim{\ \lower4pt\hbox{$
\buildrel{\displaystyle <}\over\sim$}\ }
\def\gessim{\ \lower4pt\hbox{$\buildrel{\displaystyle >}
\over\sim$}\ }

\newcommand{\e}{\mathbb{E}}
\newcommand{\p}{\mathbb{P}}

\newcommand{\vsi}{{\vec{\sigma}}}

\newcommand{\vtau}{\vec{\tau}}

\newtheorem{lemma}{\bf Lemma}
\newtheorem{definition}{\bf Definition}
\newtheorem{theorem}{\bf Theorem}

\newtheorem{example}{\bf Example}

\newtheorem{proposition}{\bf Proposition}

\newenvironment{Proof of lemma}{\noindent{\bf Proof of Lemma}}{\hfill{\tiny${\square}$}\newline}
\newenvironment{Proof of theorem}{\noindent{\bf Proof of Theorem}}{\hfill{\tiny${\square}$}\newline}
\newenvironment{Proof of theorems}{\noindent{\bf Proof of Theorems}}{\hfill{\tiny${\box}$}\newline}
\newenvironment{Proof of proposition}{\noindent{\bf Proof of Proposition}}{\hfill{\tiny${\square}$}\newline}
\newenvironment{Proof of propositions}{\noindent{\bf Proof of Propositions}}{\hfill{\tiny${\square}$}\newline}
\newenvironment{Proof of exercise}{\noindent{\it Proof of Exercise:}}{\hfill{\footnotesize${\square}$}}
\newenvironment{Acknowledgements}{\noindent{\bf Acknowledgements.}}

\numberwithin{equation}{section}

\input invlat.tex

\begin{document}

\title{Chaos in the mixed even-spin models}
\author{Wei-Kuo Chen\footnote{Department of Mathematics, University of Chicago, email: wkchen@math.uchicago.edu}}

\maketitle

\begin{abstract}
We consider a disordered system obtained by coupling two mixed even-spin models together. The chaos problem is concerned with the behavior of the coupled system when the external parameters in the two models, such as, temperature, disorder, or external field, are slightly different. It is conjectured that the overlap between two independently sampled spin configurations from, respectively, the Gibbs measures of the two models is essentially concentrated around a constant under the coupled Gibbs measure. Using the extended Guerra replica symmetry breaking bound together with a recent development of controlling the overlap using the Ghirlanda-Guerra identities as well as a new family of identities, we present rigorous results on chaos in temperature. In addition, chaos in disorder and in external field are addressed.
\end{abstract}


\section{Introduction and main results}\label{sec:intro}

The chaos problem is a very old one in the spin glass theory. It arose from the discovery that in some models, a small change in the external parameters, such as temperature, disorder, or external field, will result in a dramatic change to the overall energy landscape. Furthermore, it may as well change the location of the ground state and the organization of the pure states of the Gibbs measure. It has received a lot of attention and been intensively studied in the context of various models in physics literature in the past decades (e.g. see \cite{Rizzo09} for a recent review). In recent years, several mathematical results also have  been obtained in the problems of chaos in external field and in disorder: An example of chaos in external field for the spherical Sherrington-Kirkpatrick model was given in \cite{Pan+Tal07}. Chaos in disorder for mixed even-spin models and without external field was considered in \cite{Chatt08}, \cite{Chatt09} and a more general situation in the presence of external field was handled in \cite{Chen11}. 

\smallskip

According to physicists' viewpoint \cite{BM02,RC03,KK07}, chaos in temperature presents intricate difficulties that are very hard to be analyzed both theoretically and experimentally, mainly because this effect is exceedingly small in the perturbation theory. So far mathematically rigorous results are still very scarce. To the best of our knowledge, the only known result is a ``weak'' form of chaos in temperature studied in the mixed $p$-spin models \cite{CP12}. In this paper, we will focus on the mixed even-spin model and investigate its chaos problem in temperature. Using Guerra's replica symmetry breaking bound combining with a recent development on the control of the overlap using the Ghirlanda-Guerra identities and a new family of identities, we will present mathematically rigorous results for chaos in temperature. In addition, more general cases of chaos in disorder are considered and results of chaos in external will be addressed. 

\smallskip

We introduce the mixed even-spin model and the formulation of the chaos problem as follows. Let $\vec{\beta}=(\beta_{p})_{p\geq 1}$ be a nonnegative sequence of real numbers with $\sum_{p\geq 1}2^{2p}\beta_{p}^2<\infty$. To avoid triviality, throughout this paper, we assume that $\beta_{p}\neq 0$ for at least one $p\geq 1.$ Let $h$ be a sub-Gaussian r.v., i.e., $\e \exp th\leq d_1\exp d_2t^2$ for all $t\in\mathbb{R}.$ The most interesting examples are when $h$ is either constant or Gaussian. Given $N\geq 1,$ we consider a family of i.i.d. standard Gaussian r.v. 
\begin{align}\label{intro:eq12}
\mathcal{G}=(g_{i_1,\ldots,i_{2p}}:\forall 1\leq i_1,\ldots,i_{2p}\leq N,\,\forall p\geq 1)
\end{align}
and a family of i.i.d. copies of $h,$ $(h_i)_{i\leq N}.$ These two families of r.v. are independent of each other. The pure $2p$-spin Hamiltonians $X_{N,p}(\vsi)$ for $p\geq 1$ indexed by $\vsi\in\Sigma_N:=\{-1,+1\}^N$ is defined as
\begin{align}
\label{intro:eq17}
X_{N,p}(\vsi)&=\frac{1}{N^{p-1/2}}\sum_{1\leq i_1,\ldots,i_{2p}\leq N}g_{i_1,\ldots,i_{2p}}\sigma_{i_1}\cdots\sigma_{i_{2p}}.
\end{align}
The mixed even-spin Hamiltonian is defined as a linear combination,
\begin{align}\label{intro:eq8}
H_N(\vsi)&=X_N(\vsi)+\sum_{i\leq N}h_i\sigma_i,
\end{align}
where
\begin{align}
\label{intro:eq10}
X_N(\vsi)&:=\sum_{p\geq 1}\beta_{p}X_{N,p}(\vsi).
\end{align}
In physics, the sequence $\vec{\beta}$ is called the (inverse) temperature parameters, the family of r.v. $\mathcal{G}$ is called the disorder of the system, and $h$ is called the external field. The covariance of the Gaussian process $X_N$ can be easily computed as 
\begin{align*}
\e X_N(\vsi^1)X_N(\vsi^2)=N\xi(R(\vsi^1,\vsi^2)),
\end{align*}
where the quantity $R(\vsi^1,\vsi^2):=N^{-1}\sum_{i\leq N}\sigma_i^1\sigma_i^2$ is called the overlap between two spin configurations $\vsi^1,\vsi^2\in\Sigma_N$ and $\xi(x):=\sum_{p\geq 1}\beta_{p}^2x^{2p}.$ We define the Gibbs measure $G_N$ on $\Sigma_N$ by
\begin{align}\label{add:eq3}
G_N(\vsi)=\frac{\exp H_N(\vsi)}{Z_N},
\end{align}
where the normalizing factor $Z_N$ is called the partition function. An important case of this model is the famous Sherrington-Kirkpatrick (SK)  model \cite{SK75}, where $\beta_{1,p}=\beta_{2,p}=0$ for all $p\geq 2.$ Now consider two independently sampled spin configurations $\vsi^1$ and $\vsi^2$ from $G_N.$ It is well-known that under the measure $\e G_N\times G_N$, the overlap $R(\vsi^1,\vsi^2)$ is essentially concentrated around a constant in some part of the high temperature region, where this region is defined as the set of all temperature parameters such that the infimum in the Parisi formula (see Subsection \ref{model} below) is achieved by a Dirac measure. While in the low temperature region, i.e., outside the high temperature region, the overlap is lack of self-averaging property \cite{Pan11:2} and is conjectured to have a nontrivial weak limit, called the Parisi measure. One typical way of measuring the instability of this spin system occurred by the change of external parameters is to sample independently $\vsi$ from $G_N$ and $\vtau$ from a new Gibbs measure $G_N'$ using a perturbed external parameters from $G_N$ and consider the behavior of the overlap $R(\vsi,\vtau)$ under $\e G_N\times G_N'.$ The phenomenon of chaos states that this overlap behaves very differently and is indeed concentrated near a constant no matter that the two systems $G_N$ and $G_N'$ are in the high or low temperature regime. This is precisely the statement that we will be proving in the paper under some mild assumptions on the external parameters. 

\smallskip

Let us define two mixed even-spin models and specify their external parameters in the following. Recall $\mathcal{G}$ from $(\ref{intro:eq12}).$ Let $\mathcal{G}^1$ and $\mathcal{G}^2$ be two copies of $\mathcal{G}$ such that they together form a jointly Gaussian process. In addition, for every $1\leq i_1,\ldots,i_{2p}\leq N$ and $p\geq 1,$ the pair $(g_{i_1,\ldots,i_{2p}}^1,g_{i_1,\ldots,i_{2p}}^2)$ is independent of each other and 
\begin{align*}
\e g_{i_1,\ldots,i_{2p}}^1g_{i_1,\ldots,i_{2p}}^2=t_{p}\in[0,1].
\end{align*}
Let $h^1$ and $h^2$ be two sub-Gaussian r.v. (not necessarily independent) that do not depend on $\mathcal{G}^1$ and $\mathcal{G}^2$. Let $(h_i^1,h_i^2)$ be i.i.d. copies of $(h^1,h^2)$ for $1\leq i\leq N$ independent of $\mathcal{G}^1$ and $\mathcal{G}^2$. We consider two mixed even-spin models with Gibbs measures $G_N^1$ and $G_N^2$ corresponding to the Hamiltonians $H_N^1(\vsi)$ and $H_N^2(\vtau)$ as in $(\ref{intro:eq8})$,
\begin{align}
\begin{split}
\label{intro:eq18}
H_{N}^1(\vsi)&=X_N^1(\vsi)+\sum_{i\leq N}h_i^1\sigma_i,\\
H_N^2(\vtau)&=X_N^2(\vtau)+\sum_{i\leq N}h_i^2\tau_i,
\end{split}
\end{align} 
where  $X_N^1(\vsi)$ and $X_N^2(\vtau)$ are defined similarly as $(\ref{intro:eq10})$ for $\vsi,\vtau\in\Sigma_N$ using $\vec{\beta}_1,$ $\mathcal{G}^1$ and $\vec{\beta}_2$, $\mathcal{G}^2$, respectively. Then the covariances of $X_N^1$ and $X_N^2$ can be easily computed as
\begin{align*}
\e X_N^1(\vsi^1)X_N^1(\vsi^2)&=N\xi_{1,1}(R(\vsi^1,\vsi^2)),\\
\e X_N^2(\vtau^1)X_N^2(\vtau^2)&=N\xi_{2,2}(R(\vtau^1,\vtau^2)),\\
\e X_N^1(\vsi^1)X_N^2(\vtau^1)&=N\xi_{1,2}(R(\vsi^1,\vtau^1)),
\end{align*}
for $\vsi^1,\vsi^2,\vtau^1,\vtau^2\in\Sigma_N,$ where
\begin{align}\label{add:eq2}
\xi_{1,1}(x)&=\sum_{p\geq 1}\beta_{1,p}^2 x^{2p},\,\,
\xi_{2,2}(x)=\sum_{p\geq 1}\beta_{2,p}^2 x^{2p},\,\,
\xi_{1,2}(x)=\sum_{p\geq 1}t_{p}\beta_{1,p}\beta_{2,p}x^{2p}.
\end{align} We will denote by $(\vsi^{\ell},\vtau^{\ell})_{\ell\geq 1}$ an i.i.d. sequence of replicas from the measure $G_N^1\times G_N^2$ and by $\left<\cdot\right>$ the Gibbs average with respect to $(G_N^1\times G_N^2)^{\otimes \infty}.$ One may regard the pair of Gibbs measures $(G_N^1,G_N^2)$ as $(G_N,G_N')$ mentioned above.

\smallskip

Now we present the main results. The first one is concerned with chaos in temperature, namely, the behavior of the overlap $R(\vsi,\vtau)$ between two systems $G_N^1$ and $G_N^2$ at different temperatures $\vec{\beta}_1$ and $\vec{\beta}_2$. For $j=1,2,$ let $$
\mathcal{I}_j=\{p\in\mathbb{N}:\beta_{j,p}\neq 0\}.
$$
We introduce a family of subsets of natural numbers,
\begin{align*}
\mathcal{C}_0&=\{\mathcal{I}\subseteq\mathbb{N}:\mbox{linear span of $1$ and $(x^{2p})_{p\in\mathcal{I}}$ is dense in $(C[0,1],\|\cdot\|_\infty)$}\}.
\end{align*}
Note that the M\"{u}ntz-Szasz theorem (see Theorem 15.26 \cite{Rudin}) provides a very simple criterion: $\mathcal{I}\in \mathcal{C}_0$ if and only if $\sum_{p\in\mathcal{I}}1/p=\infty.$ Define the following mild conditions on the temperature parameters $\vec{\beta}_1$ and $\vec{\beta}_2:$
\begin{itemize}
\item[$(C_1)$] There exist $\mathcal{A}\subseteq \mathcal{I}_1$ and $p_0\in \mathcal{I}_1\setminus \mathcal{A}$ such that $\mathcal{A}\in\mathcal{C}_0$ and for some $\nu\in\mathbb{R}$ we have $\beta_{2,p}=\nu\beta_{1,p}$ for all $p\in\mathcal{A}$ and $\beta_{2,p_0}\neq \nu\beta_{1,p_0}.$
\item[$(C_2)$] There exist $\mathcal{A}\subseteq \mathcal{I}_2$ and $p_0\in \mathcal{I}_2\setminus \mathcal{A}$ such that $\mathcal{A}\in\mathcal{C}_0$ and for some $\nu\in\mathbb{R}$ we have $\beta_{1,p}=\nu\beta_{2,p}$ for all $p\in\mathcal{A}$ and $\beta_{1,p_0}\neq \nu\beta_{2,p_0}.$
\end{itemize}
Two important examples of $\vec{\beta}_1$ and $\vec{\beta}_2$ satisfying both conditions $(C_1)$ and $(C_2)$ are that we add higher order spin interactions to the SK models and perturb either the SK temperatures or the higher order spin interaction temperatures at the same rate:
 
\begin{example}
$\beta_{1,p},\beta_{2,p}\neq 0$ for all $p\geq 1$ with $\beta_{1,1}\neq\beta_{2,1}$ and $\beta_{1,p}=\beta_{2,p}$ for all $p\geq 2$.
\end{example}

\begin{example}
$\beta_{1,p},\beta_{2,p}\neq 0$ for all $p\geq 1$ with $\beta_{1,1}=\beta_{2,1}$ and for some $\nu\neq 1,$ $\beta_{1,p}=\nu\beta_{2,p}$ for all $p\geq 2$.
\end{example}

\begin{theorem}[Temperature chaos]\label{thm:tc} 
Let $h^1$ and $h^2$ be jointly Gaussian. Suppose that $t_{p}=1$ for all $p\geq 1$ and that $\mathcal{I}_1$ and $\mathcal{I}_2$ satisfy $(C_1)$ and $(C_2)$, respectively. If $\e (h^j)^2=0$ for $j=1$ or $2$, then 
\begin{align}
\label{thm:tc:eq1}
\lim_{N\rightarrow\infty}\e\left<I(|R(\vsi,\vtau)|>\varepsilon)\right>=0,\,\,\forall \varepsilon>0.
\end{align}
If $\mbox{\rm Var}(h^j)\neq 0$ for both $j=1,2$, then there exists some constant $u_f$ such that
\begin{align}
\label{thm:tc:eq2}
\lim_{N\rightarrow\infty}\e\left<I(|R(\vsi,\vtau)-u_f|>\varepsilon)\right>=0,\,\,\forall \varepsilon>0.
\end{align}
\end{theorem}

\smallskip

 Theorem \ref{thm:tc} is the first rigorous chaos result in temperature in the mixed even-spin model. It indicates the sensitivity of the model to the change of temperatures. However, due to technical reasons, it remains unknown how to verify $(\ref{thm:tc:eq1})$ and $(\ref{thm:tc:eq2})$ in the setting of the SK model, i.e., $\beta_{1,1}\neq \beta_{2,1}$ and $\beta_{1,p}=\beta_{2,p}=0$ for all $p\geq 2,$ which seems by far the most interesting case to physicists. Also the situation of constant external fields is unclear. Let us remark that the constant $u_f$ in Theorem \ref{thm:tc} as well as in Theorems \ref{thm:dc} and \ref{thm:efc} below could possibly be equal to zero. For instance, as one will see in Proposition $\ref{intro:prop1}$, if $h^1$ and $h^2$ are independent and symmetric with respect to the origin, then $u_f=0.$ The determination of $u_f$ is a highly technical issue. It is indeed the unique fixed point of a function related to Parisi's formula and measures that will be discussed in Section $\ref{sec2}$.

\smallskip

Next, let us turn to the main results on chaos in disorder. In this problem, we want to know the behavior of the overlap in the coupled system occurred by the change of the disorders. The first rigorous study of this problem without external field was given in \cite{Chatt09} and later more general situations of the models with external fields were handled in \cite{Chen11}. In view of the arguments therein, for technical purposes, the Hamiltonians for the two systems are assumed to be identically distributed. We prove that chaos in disorder is still valid even when two Hamiltonians do not have the same distribution if some mild conditions on the temperature parameters are added. The following is our main result.

\begin{theorem}[Disorder chaos]\label{thm:dc} Let $h^1$ and $h^2$ be jointly Gaussian r.v. Suppose that $0\leq t_{p}<1$ for some $p\in \mathcal{I}_1\cap\mathcal{I}_2.$ If $\e (h^j)^2=0$ and $\mathcal{I}_j\in\mathcal{C}_0$ for $j=1$ or $2$, then $(\ref{thm:tc:eq1})$ holds. If $\mbox{\rm Var} (h^j)\neq 0$ and $\mathcal{I}_j\in\mathcal{C}_0$ for both $j=1,2$, then $(\ref{thm:tc:eq2})$ holds.
\end{theorem}

Lastly, suppose that the two systems use the same temperature parameters and disorders, i.e., $\vec{\beta}_1=\vec{\beta}_2$ and $t_{p}=1$ for all $p\geq 1.$ We would like to know how the overlap $R(\vsi,\vtau)$ in the coupled system is influenced when the external fields $h^1$ and $h^2$ are essentially different from each other. To begin with, let us give a counterexample to illustrate that the chaotic property does not always hold for arbitrary choices of $h^1$ and $h^2.$ For instance, one may consider $h^1$ and $h^2$ having the relation $h^1=-h^2$. Since $\xi_{1,1}=\xi_{2,2}=\xi_{1,2}$ is even, one may check easily that
$$(X_N^1(\vsi)+X_N^2(\vtau):\vsi,\vtau\in\Sigma_N)$$ 
and 
$$(X_N^1(\vsi)+X_N^2(-\vtau):\vsi,\vtau\in\Sigma_N)$$ 
have the same joint distribution. Thus, using $h^1=-h^2$ and change of variables $-\vtau\rightarrow\vtau,$ for every Borel measurable subset $A$ of $[-1,1],$ we obtain
\begin{align*}
&\e G_N^1\times G_N^2(\{(\vsi,\vtau):R(\vsi,\vtau)\in A\})=\e G_N^1\times G_N^1(\{(\vsi^1,\vsi^2):R(\vsi^1,\vsi^2)\in- A\}),
\end{align*}
where $-A:=\{-x:x\in A\}.$ As we have mentioned before, since the limiting distribution of the overlap under $\e G_N^1\times G_N^1$ is nontrivial in the low temperature regime, we can not witness chaos in this case. Thus, in order to capture the chaotic feature, further assumptions on the external fields are required. The theorem below provides one possible choice of $(h^1,h^2)$ by assuming that they are different in disorder.

\begin{theorem}[External field chaos]\label{thm:efc} 
Suppose that $\vec{\beta}_1=\vec{\beta}_2$ and $t_{p}=1$ for all $p\geq 1.$
Let $h^1$ and $h^2$ be two r.v. having the same sub-Gaussian distribution. If $\e (h^1\pm h^2)^2\neq 0,$ then for any $\varepsilon>0,$ there exists some positive constant $K$ independent of $N$ such that for all $N\geq 1,$
\begin{align}
\label{thm:efc:eq2}
\e\left<I(|R(\vsi,\vtau)-u_f|\geq\varepsilon)\right><K\exp\left(-\frac{N}{K}\right)
\end{align}
for some constant $u_f$.
\end{theorem}


\smallskip

Apparently, $(\ref{thm:efc:eq2})$ is much stronger a chaos result comparing to those in Theorems $\ref{thm:tc}$ and $\ref{thm:dc}$. The main reason  will be illustrated in our proof sketches for which we are going to discuss now. This paper is mainly motivated by two recent works \cite{Chen11} and \cite{CP12}. We consider the coupled free energy,
\begin{align}
\label{sec0:eq1}
p_{N,u}&:=\frac{1}{N}\e\log \sum_{R(\vsi,\vtau)=u}\exp\left(H_N^1(\vsi)+H_N^2(\vtau)\right)
\end{align}
for $u\in S_N:=\{i/N:-N\leq i\leq N\}$ and analyze this quantity via an extended Guerra replica symmetry breaking bound. Suppose that $\mu_P^1$ and $\mu_P^2$ are the Parisi measures (see Definition \ref{def2} below) corresponding to the two systems, respectively. Set $c_1=\min\mbox{supp}\mu_P^1$ and $c_2=\min\mbox{supp}\mu_P^2.$ We show that this bound naturally gives rise to a function that determines $u_f$ and implies the statement that for any $\varepsilon>0$, there exists some $\varepsilon^*>0$ such that if $N$ is sufficiently large, then
\begin{align}\label{add4}
p_{N,u}\leq \frac{1}{N}\e\log Z_N^1+\frac{1}{N}\e\log Z_N^2-\varepsilon^*
\end{align}
for all $u\in S_N$ with $|u|\leq \sqrt{v_1v_2}$ and $|u-u_f|\geq \varepsilon,$ where $v_1>c_1$ and $v_2>c_2$ are two constants independent of $N$ and $Z_N^1$ and $Z_N^2$ are the partition functions of the two systems. From this, a standard application of concentration of measure for disorders $\mathcal{G}^1,\mathcal{G}^2$ and external fields $(h_i^1)_{1\leq i\leq N},(h_i^2)_{1\leq i\leq N}$ means
\begin{align}
\label{add5}
\e\left<I(|R(\vsi,\vtau)|\leq \sqrt{v_1v_2},\,|R(\vsi,\vtau)-u_f|\geq \varepsilon)\right>\leq K\exp\left(-\frac{N}{K}\right)
\end{align}
for all $N\geq 1,$ where $K>0$ is a constant independent of $N.$ One would like to expect that using appropriate choices of parameters for Guerra's bound also implies $(\ref{add4})$ for all $u\in S_N$ with $|u|\geq \sqrt{v_1v_2}$ and again from concentration of measure, 
\begin{align}
\label{add6}
\e\left<I(|R(\vsi,\vtau)|\geq \sqrt{v_1v_2})\right>\leq K\exp\left(-\frac{N}{K}\right)
\end{align}
for all $N\geq 1.$ This together with $(\ref{add5})$ will then yield an exponential bound as (\ref{thm:efc:eq2}). It turns out that $(\ref{add6})$ can be successfully carried out and will be our main approach to the problem of chaos in external field as in Theorem $\ref{thm:efc}$, which relies heavily on the fact that the Hamiltonians in the two systems are identically distributed. Unfortunately, in the setting chaos in temperature or disorder, this fact is generally not valid that creates highly intractable difficulties of choosing parameters in Guerra's bound. To resolve this technical issue, we will adapt another approach \cite{CP12} by considering the Ghirlanda-Guerra identities and developing a new family of identities for the coupled system under the mild assumptions on the temperature parameters and disorders. These two families of identities contain the information about how the spin configurations between two systems interact with each other that allows us to control the overlap $R(\vsi,\vtau)$ between the two systems by using the overlaps $R(\vsi^1,\vsi^2)$ and $R(\vtau^1,\vtau^2)$ from the individual systems. Ultimately they lead to a weak result,
\begin{align*}
\lim_{N\rightarrow\infty}\e\left<I(|R(\vsi,\vtau)|\geq \sqrt{v_1v_2})\right>=0.
\end{align*}
This and $(\ref{add5})$ together imply the conclusions of Theorems \ref{thm:tc} and \ref{thm:dc}. 

\smallskip

The rest of the paper is organized as follows. Section \ref{sec3} is devoted to studying some basic properties of Parisi's measures that are needed in our chaos results. In particular, we prove that in the absence of external field, the supports of Parisi's measures always contain the origin for all temperatures. The central proof of this result is played by a fundamental fixed point theorem that will be used later to determine the constant $u_f$ in Section \ref{sec2}. Section \ref{sec1} begins by recalling the extended Guerra's replica symmetry breaking bound for the coupled free energy $(\ref{sec0:eq1})$. We will choose parameters for this bound to derive a manageable bound for the coupled free energy in terms of a function $\phi_{v_1,v_2}$ as in $(\ref{chaos:eq1})$ below and the Parisi formulas for the two systems. In Section \ref{sec2}, we will investigate the behavior of the overlap $R(\vsi,\vtau)$ inside the interval $[-\sqrt{c_1c_2},\sqrt{c_1c_2}]$. The argument relies on the fixed point theorem established in Section \ref{sec3} that allow us to determine $u_f$ and to derive an exponentially tail control $(\ref{add5})$. In Section \ref{identity}, we demonstrate how to use the given conditions in Theorems \ref{thm:tc} and \ref{thm:dc} to derive the Ghirlanda-Guerra identities and a new family of identities for the coupled system. They together with an application of the Cauchy-Schwartz inequality provide an approach to controlling the overlap $R(\vsi,\vtau)$ via $R(\vsi^1,\vsi^2)$ and $R(\vtau^1,\vtau^2)$. Finally, we combine all results in every section to prove Theorems \ref{thm:tc}, \ref{thm:dc}, and \ref{thm:efc} in Section 6. 

\smallskip
\smallskip

\begin{Acknowledgements}
The author would like to thank Antonio Auffinger, Wan-Ching Hu, and the anonymous referee for their careful reading and several suggestions regarding the presentation of the paper.
\end{Acknowledgements}

\section{Properties of Parisi's measures}\label{sec3}

The Parisi formula and measures are intimately related to the investigation of chaos problem as they will induce a crucial function that determines the location at where the overlap $R(\vsi,\vtau)$ in the coupled system is concentrated, see Sections \ref{sec1} and \ref{sec2} below. In this section, we will recall their definitions and study the support of the Parisi measures. 

\subsection{Parisi's formula and measures}\label{model}

For given temperature $\vec{\beta}$ and external field $h,$ recall the Gibbs measure $G_N$ and partition function $Z_N$ from $(\ref{add:eq3}).$ In statistical physics, the thermodynamic limit of the free energy 
$$p_N:=\frac{1}{N}\e\log Z_N$$ 
is one of the most important quantities that describes the macroscopic behavior of the system. It can be computed by the famous Parisi formula described below. For any given integer $k\geq 0,$ let $\mathbf{m}=(m_p)_{0\leq p\leq k+1}$ and $\mathbf{q}=(q_p)_{0\leq p\leq k+2}$ be real numbers satisfying
\begin{align}
\begin{split}\label{intro:eq0}
m_0&=0\leq m_1\leq \cdots\leq m_k\leq m_{k+1}=1,\\
q_0&=0\leq q_1\leq \cdots\leq q_{k+1}\leq q_{k+2}=1.
\end{split}
\end{align}
One may think of this triplet $(k,\mathbf{m},\mathbf{q})$ as a probability measure $\mu$ on $[0,1]$ that has all of its masses concentrated at a finite number of points $q_1,\ldots,q_{k+1}$ and $\mu([0,q_p])=m_p$ for $0\leq p\leq k+1.$ Let $z_0,\ldots,z_{k+1}$ be independent centered Gaussian r.v. with $\e z_p^2=\xi'(q_{p+1})-\xi'(q_p)$ for $0\leq p\leq k+1.$ Starting with
$$
X_{k+2}=\log \cosh\left(h+\sum_{0\leq p\leq k+1}z_p\right),
$$
we define decreasingly for $1\leq p\leq k+1,$ $$
X_p=\frac{1}{m_p}\log \e_p\exp m_p X_{p+1},
$$
where $\e_p$ means the expectation on the r.v. $z_p,z_{p+1},\ldots,z_{k+1}.$ If $m_p=0$ for some $p,$ we define $X_p=\e_pX_{p+1}.$ Finally, we define $X_0=\e X_1.$ Set
\begin{align}
\label{intro:eq5}
\mathcal{P}_k(\mathbf{m},\mathbf{q})=\log 2+X_0-\frac{1}{2}\sum_{p=1}^{k+1}m_p(\theta(q_{p+1})-\theta(q_p)),
\end{align}
where $\theta(x):=x\xi'(x)-\xi(x).$ The importance of this quantity lies on the fact that it yields Guerra's bound for the free energy \cite{Guerra03},
\begin{align}\label{intro:eq6}
p_N\leq \mathcal{P}_k(\mathbf{m},\mathbf{q})
\end{align} 
for any given triplet $(k,\mathbf{m},\mathbf{q}).$ Usually, we call $\mathcal{P}_k(\mathbf{m},\mathbf{q})$ the replica symmetry bound if $k=0$ and the $k$-th level replica symmetry breaking bound if $k\geq 1$ and all strict inequalities in \eqref{intro:eq0} hold. Set
\begin{align}
\label{intro:eq13}
\mathcal{P}(\xi,h)=\inf_{(k,\mathbf{m},\mathbf{q})}\mathcal{P}_k(\mathbf{m},\mathbf{q}),
\end{align}
where the infimum is taken over all triplets. The Parisi formula states that the thermodynamic limit of the free energy is given by the variational formula $(\ref{intro:eq13})$,
\begin{align}\label{pf}
\lim_{N\rightarrow\infty}p_N=\mathcal{P}(\xi,h).
\end{align}
This formula was first rigorously verified in \cite{Tal06}. Note that $\mathcal{P}_{k}(\mathbf{m},\mathbf{q})=\mathcal{P}_{k'}(\mathbf{m}',\mathbf{q}')$ if two triplets $(k,\mathbf{m},\mathbf{q})$ and $(k',\mathbf{m}',\mathbf{q}')$ induce the same probability distribution. Using this, we can define a functional $\mathcal{P}(\xi,h,\cdot)$ on the space of all probability measures on $[0,1]$, that consist of a finite number of jumps, by letting $\mathcal{P}(\xi,h,\mu)=\mathcal{P}_k(\mathbf{m},\mathbf{q})$ if $\mu$ corresponds to the triplet $(k,\mathbf{m},\mathbf{q}).$ It is well-known \cite{Guerra03} that this functional is Lipschitz continuous with respect the metric $d(\mu,\mu')=\int_0^1|\mu([0,x])-\mu'([0,x])|dx.$ Thus, we can extend $\mathcal{P}(\xi,h,\cdot)$ continuously to the space of all probability measures on $[0,1]$ and for simplicity, we will still denote this extension by $\mathcal{P}(\xi,h,\cdot).$ This then allows us to replace the infimum in the Parisi formula by taking minimum over all probability measures on $[0,1]$. 

\begin{definition}
\label{def1} Let $\mu$ be a probability measure corresponding to the triplet $(k,\mathbf{m},\mathbf{q})$. Given $\varepsilon>0$, we say that $\mu$ satisfies condition MIN$(\varepsilon)$ if the following occurs. First, the sequences $\mathbf{m}=(m_p)_{0\leq p\leq k+1}$ and $\mathbf{q}=(q_p)_{0\leq p\leq k+2}$ satisfy
\begin{align*}
m_0&=0<m_1<\cdots<m_k<m_{k+1}=1\\
q_0&=0\leq q_1<\cdots<q_{k+1}<q_{k+2}=1.
\end{align*}
In addition, 
$$
\mathcal{P}_k(\mathbf{m},\mathbf{q})\leq \mathcal{P}(\xi,h)+\varepsilon
$$
and 
$$
\mathcal{P}_{k}(\mathbf{m},\mathbf{q})\,\,\mbox{realizes the minimum of $\mathcal{P}_k$ over all choices of $\mathbf{m}$ and $\mathbf{q}$}.
$$
\end{definition}

Let us remark that for any given $\varepsilon>0$, one can always find a $\mu\in\mbox{MIN}(\varepsilon)$ by Lemma 14.5.5 and Proposition 14.7.5 in \cite{Tal11}. In addition, if $\e h^2\neq 0,$ then one further has $q_1>0$. As one might expect, there may have several minimizers to Parisi's formula. Among possibly many minimizers, we are particularly interested in those, called the Parisi measures $\mu_P$ defined below. 

\begin{definition}
\label{def2}
A probability measure $\mu_P$ is called a Parisi measure (corresponding to the function $\xi$ and external field $h$) if it is the weak limit of a sequence of probability measures $\mu_n\in \mbox{MIN}(\varepsilon_n)$ for some sequence of real numbers $(\varepsilon_n)_{n\geq 1}$ with $\varepsilon_n\downarrow 0.$
\end{definition}

\smallskip

In physics, it is conjectured that the Parisi measure is unique and it is the limiting distribution of the overlap. Under suitable technical assumption on $\vec{\beta}$, such as $\beta_p\neq 0$ for all $p\geq 1,$ these statements are verified to be valid, but the general situation remains open. 

\smallskip

There are two basic properties about the Parisi measures and the overlap that are of great importance and are intimately related to the study of chaos phenomena. First, in the presence of external field, $\e h^2\neq 0,$ they satisfy a positivity principle, namely, for any Parisi measure $\mu_P$, we have that 
\begin{align}
\label{intro:eq20}
c:=\inf\mbox{supp}\mu_P>0
\end{align} 
and for all $c'<c,$
\begin{align}
\label{intro:eq19}
\e G_N\times G_N(\{(\vsi^1,\vsi^2):R(\vsi^1,\vsi^2)\leq c'\})\leq K\exp \left(-\frac{N}{K}\right)
\end{align}
for all $N\geq 1,$ where $K$ is a positive constant independent of $N.$ This result can be found in Section 14.12 \cite{Tal11}. The second property is concerned with their behavior in the absence of external field, $\e h^2=0$. It is believed according to physicists' numerical simulations \cite{MPV} that in this case the origin is contained in the support of the limiting distribution of the overlap. It turns out that mathematically there are several approaches to verify this observation in the high temperature regime (see Chapter 1 \cite{Tal11}) but it is by no means clear how to attack this problem in the low temperature regime. In this paper, we prove that at least this observation is true for the Parisi measures. Below is the statement of our main result.

\begin{theorem}
\label{sec3:thm1}
For any $\vec{\beta},$ if $\e h^2=0,$ then $0\in\mbox{supp}\mu_P$ for every Parisi measure $\mu_P.$
\end{theorem}

Although in this paper Theorem \ref{sec3:thm1} will only be used to derive our chaos results, it is also of independent interest in understanding the structure of the pure states of the Gibbs measure. Let us remark that in the spherical Sherrington-Kirkpatrick model without external field \cite{Pan+Tal07}, the Parisi measure consists of a single point mass at some $c>0$ on the low temperature regime, which is very different to our result in Theorem \ref{sec3:thm1}. 

\subsection{An auxiliary function and a fixed point theorem}

The central rhythm of the proof for Theorem \ref{sec3:thm1} and our results on chaos are played by an auxiliary function and a fixed point theorem given below. Suppose that $\mu$ is a probability measure corresponding to some $(k,\mathbf{m},\mathbf{q})$. Recall $X_0$ from $(\ref{intro:eq5})$ by using this triplet. A very nice property about this quantity says that it can also be computed as
$\e \Phi_\mu(h,0)$, where $\Phi_\mu:\mathbb{R}\times [0,1]\rightarrow\mathbb{R}$ is the solution to the following PDE,
\begin{align}
\label{intro:eq14}
\frac{\partial\Phi_\mu}{\partial q}=-\frac{\xi''(q)}{2}\left(\frac{\partial^2\Phi_\mu}{\partial x^2}+\mu([0,q])\left(\frac{\partial\Phi_\mu}{\partial x}\right)^2\right)
\end{align}
with $\Phi_\mu(x,1)=\log\cosh x.$ Let $\mu_P$ be a Parisi measure and $(\mu_n)_{n\geq 1}$ be any sequence of probability measures consisted of a finite number of point masses that converges weakly to $\mu_P$ . Let $\Phi_{\mu_n}$ be the PDE solution $(\ref{intro:eq14})$ associated to $\mu_n.$ From the Lipschitz property of the Parisi functional, one sees that $(\Phi_{\mu_n})_{n\geq 1}$ converges uniformly on $\mathbb{R}\times[0,1]$. Define \begin{align}
\label{intro:eq16}
\Phi_{\mu_P}=\lim_{n\rightarrow\infty}\Phi_{\mu_n}.
\end{align} 
Note that from the Lipschitz property of the Parisi function, $\Phi_{\mu_P}$ is indeed independent of the choice of the sequence $(\mu_n)_{n\geq 1}.$ Let us summarize some further properties about $\Phi_{\mu_P}$ that will be used throughout the paper in the following proposition. 

\begin{proposition}\label{add:prop1} The following facts hold for $\Phi_{\mu_P}.$
\begin{enumerate}
\item[$(a)$] For $0\leq j\leq 3,$ $\lim_{n\rightarrow\infty}\frac{\partial^j \Phi_{\mu_n}}{\partial x^j}=\frac{\partial^j \Phi_{\mu_P}}{\partial x^j}$ uniformly on $\mathbb{R}\times[0,1]$.
\item[$(b)$] $\|\frac{\partial \Phi_{\mu_P}}{\partial x}\|_\infty\leq 1$, $0\leq\inf_x\frac{\partial^2 \Phi_{\mu_P}}{\partial x^2}<\|\frac{\partial^2 \Phi_{\mu_P}}{\partial x^2}\|_\infty\leq 1$, and $\|\frac{\partial^3 \Phi_{\mu_P}}{\partial x^3}\|_\infty\leq 4$. 
\item[$(c)$] $\frac{\partial \Phi_{\mu_P}}{\partial x}(\cdot,q)$ is odd for any $q.$ 
\item[$(d)$] If $\e h^2\neq 0,$ recalling from $(\ref{intro:eq20})$, we have that $c>0$ and
\begin{align*}
\e\left(\frac{\partial\Phi_{\mu_P}}{\partial x}(h+\chi,c)\right)^2&=c,\\
\xi''(c)\e\left(\frac{\partial^2\Phi_{\mu_P}}{\partial x^2}(h+\chi,c)\right)^2&\leq 1,
\end{align*}
where $\chi$ is centered Gaussian with $\e \chi^2=\xi'(c)$ and is independent of $h.$ 
\end{enumerate}
\end{proposition}

\begin{proof}
The proofs for $j=0,1$ in the first statement are given in the proof of Theorem 3.2 \cite{Tal07}. One may see that indeed a similar but much more tedious argument as Theorem 3.2 \cite{Tal07} will also yield the cases for $j=2,3.$ The second and third assertions are concluded from the statements and proofs of Lemma 14.7.16 \cite{Tal11} and Lemma 2 \cite{Chen11}. Finally, the fourth statement is exactly Lemma 12 \cite{Chen11}.
\end{proof}

Recall $\xi_{1,1},$ $\xi_{2,2}$, and $\xi_{1,2}$ from (\ref{add:eq2}). Let $v_1,v_2$ be two real numbers satisfying $0<v_1,v_2\leq 1.$ Observe that from the fact $t_p\in[0,1]$ and then the Cauchy-Schwartz inequality, 
\begin{align}
\begin{split}
\label{sec2.2:eq2}
\xi_{1,2}'(\sqrt{v_1v_2})&\leq \sum_{p\geq 1}{2p}\beta_{1,p}\beta_{2,p}(\sqrt{v_1v_2})^{2p-1}\\
&=\sum_{p\geq 1}\sqrt{2p}\beta_{1,p}v_1^{p-1/2}\sqrt{2p}\beta_{2,p}v_2^{p-1/2}\\
&\leq \biggl(\sum_{p\geq 1}2p\beta_{1,p}^2 v_1^{2p-1}\biggr)^{1/2}\biggl(\sum_{p\geq 1}2p\beta_{2,p}^2v_2^{2p-1}\biggr)^{1/2}\\
&=\xi_{1,1}'(v_1)^{1/2}\xi_{2,2}'(v_2)^{1/2}.
\end{split}
\end{align}
Let us note that the two sequences $\vec{\beta}_1$ and $\vec{\beta}_2$ in the definitions of $\xi_{1,1}$ and $\xi_{2,2}$ are nontrivial. This implies $\xi_{1,1}'(v_1)$, $\xi_{1,1}''(v_1)$,  $\xi_{2,2}'(v_2),$ $\xi_{2,2}''(v_2)>0$. Set a nonnegative function $$
\eta(u)=\frac{\xi_{1,2}'(|u|)}{\xi'_{1,1}(v_1)^{1/2}\xi_{2,2}'(v_2)^{1/2}}.
$$
From \eqref{sec2.2:eq2}, $\eta(u)\leq\eta(\sqrt{v_1v_2}) \leq 1$ for $|u|\leq \sqrt{v_1v_2}.$ This allows us to define jointly Gaussian r.v.
\begin{align}
\begin{split}\label{sec2.2:eq1}
\chi^1&=\chi^1(u):=\sqrt{\xi_{1,1}'(v_1)}(\sqrt{\eta(u)}w+\sqrt{1-\eta(u)}w_1),\\
\chi^2&=\chi^2(u):=\sqrt{\xi_{2,2}'(v_2)}(\mbox{sign}(u)\sqrt{\eta(u)}w+\sqrt{1-\eta(u)}w_2)
\end{split}
\end{align} 
for $|u|\leq \sqrt{v_1v_2},$ where $w,w_1,w_2$ are i.i.d. standard Gaussian. Their covariance can be computed as $\e (\chi^1)^2=\xi_{1,1}'(v_1),$ $\e(\chi^2)^2=\xi_{2,2}'(v_2)$, and $\e \chi^1\chi^2=\xi_{1,2}'(u).$ Such construction of $\chi^1$ and $\chi^2$ will be used several times throughout the paper. Below is our fixed point theorem.

\begin{theorem} \label{sec2:prop3} Suppose that $0<c_1,c_2\leq 1$. Consider the r.v. $\chi^1,\chi^2$ as defined by letting $v_1=c_1,v_2=c_2$ in \eqref{sec2.2:eq1}. Let $h^1,h^2$ be any two r.v. independent of $\chi^1,\chi^2.$ Suppose that $F_1,F_2$ are real-valued functions on $\mathbb{R}$ with $\|F_j\|_\infty,\|F_j''\|_\infty<\infty$, and $0\leq\inf F_j'<\|F_j'\|_\infty<\infty$ for $j=1,2.$ Define 
\begin{align}\label{sec2:prop3:eq1}
F(u)=\e F_1(h^1+\chi^1)F_2(h^2+\chi^2)
\end{align}
for $|u|\leq \sqrt{c_1c_2}$. If
\begin{align}
\begin{split}
\label{sec2:prop3:eq3}
\e F_j(h^j+\chi^j)^2&=c_j,
\end{split}\\
\begin{split}
\label{sec2:prop3:eq4}
\xi_{j,j}''(c_j)\e F_j'(h^j+\chi^j)^2&\leq 1,
\end{split}
\end{align}
then $F$ maps $[-\sqrt{c_1c_2},\sqrt{c_1c_2}]$ to itself and $F$ has a unique fixed point $u_f.$
\end{theorem}

\begin{proof} From the Cauchy-Schwartz inequality and \eqref{sec2:prop3:eq3},
\begin{align*}
|F(u)|\leq (\e F_1(h^1+\chi^1)^2)^{1/2}(\e F_2(h^2+\chi^2)^2)^{1/2}\leq \sqrt{c_1c_2},
\end{align*}
which gives the first assertion and this guarantees the existence of a fixed point by the intermediate value theorem. To prove the uniqueness, it suffices to prove that $|F'(u)|<1$ for all $|u|< \sqrt{c_1c_2}.$ Indeed, if this is true and $F$ has two distinct fixed points $u_f$ and $u_f',$ then the mean value theorem yields a contradiction,
$$
|u_f-u_f'|=|F(u_f)-F(u_f')|=|F'(u_0)||u_f-u_f'|<|u_f-u_f'|
$$
for some $|u_0|<\sqrt{c_1c_2}.$ Note that the first and second partial derivatives of $F_1$ and $F_2$ are uniformly bounded. From \eqref{sec2.2:eq1}, the Gaussian integration by parts leads to $$
F'(u)=\xi_{1,2}''(u)\e F_1'(h^1+\chi^1)F_2'(h^2+\chi^2),\,|u|<\sqrt{c_1c_2}.
$$
Also note that $\xi_{1,2}''$ is an even. From the Cauchy-Schwartz inequality, a similar argument as \eqref{sec2.2:eq2} gives
$
\xi_{1,2}''(u)
\leq \xi_{1,1}''(c_1)^{1/2}\xi_{2,2}''(c_2)^{1/2}
$ for $|u|<\sqrt{c_1c_2}.$
This and $(\ref{sec2:prop3:eq4})$ together with another application of the Cauchy-Schwartz inequality imply that for all $|u|<\sqrt{c_1c_2},$
\begin{align}\label{sec2:prop3:proof:eq1}
|F'(u)|&=\frac{\xi_{1,2}''(u)}{\xi_{1,1}''(c_1)^{1/2}\xi_{2,2}''(c_2)^{1/2}}\e Z_1Z_2
\leq\e Z_1Z_2\leq (\e Z_1^2)^{1/2}(\e Z_2^2)^{1/2}\leq 1,
\end{align}  
where 
$$
Z_1:=\xi_{1,1}''(c_1)^{1/2}F_1'(h^1+\chi^1)\,\,\mbox{and}\,\,Z_2:=\xi_{2,2}''(c_2)^{1/2}F_2'(h^2+\chi^2).
$$
If $|F'(u_0)|=1$ for some $|u_0|<\sqrt{c_1c_2},$ then $(\ref{sec2:prop3:proof:eq1})$ implies $\e Z_1Z_2=1.$ Thus, from $(\ref{sec2:prop3:eq4}),$ 
$$
\e(Z_1-Z_2)^2=\e Z_1^2+\e Z_2^2-2\e Z_1Z_2\leq 2-2=0
$$
and so $Z_1=Z_2.$ Now, on the one hand, since $\eta(u_0)<1$ and $\xi_{1,1}'(c_1),$ $\xi_{2,2}'(c_2)>0,$ we have $\p(h^1+\chi^1\in O_1,h^2+\chi^2\in O_2)>0$ for all open subsets $O_1,O_2.$ On the other hand, from $\inf F_j'<\sup F_j',$ $F_j'$ is not a constant function. These two facts together with $\xi_{1,1}''(c_1),$ $\xi_{2,2}''(c_2)>0$ imply that $Z_1\neq Z_2$ has nonzero probability, a contradiction. So $|F'(u)|<1$ for all $|u|<\sqrt{c_1c_2}$ and this completes our proof.
\end{proof}

\subsection{Proof of Theorem \ref{sec3:thm1}} 

To motivate our approach, we will first consider the case that $\mu_P$ is a replica symmetry solution to the Parisi formula, i.e., $\mu_P(\{c\})=1$ for some $0\leq c\leq 1.$  We then continue to study the case that $\mu_P$ is a replica symmetry breaking solution, i.e., $\mu_P$ is nontrivial. As one shall see, the argument for the second case is exactly the same as that presented in the first case. Only now added complications resulting from the more complicated structure of $\mu_P$ has to be treated subtly.

\smallskip
\smallskip
 
\begin{Proof of theorem} {\bf\ref{sec3:thm1} for replica symmetric $\mu_P$:}
Assume that $\mu_P(\{c\})=1$ for some $c\in[0,1]$. Suppose on the contrary that $0<c\leq 1.$  Recall the Parisi functional $\mathcal{P}_0(\mathbf{m},\mathbf{q})$ from $(\ref{intro:eq5})$, where $\mathbf{m}=(0,1)$ and $\mathbf{q}=(0,q,1)$ for $0\leq q\leq 1.$ First, observe that 
\begin{align*}
\frac{d\mathcal{P}_0}{dq}(\mathbf{m},\mathbf{q})=\frac{1}{2}\xi''(q)(q-\e\tanh^2(w\sqrt{\xi'(q)}))
\end{align*}
and this function is $>0$ at $q=1,$ where $w$ is standard Gaussian. We conclude that $c<1$ and thus, $c$ satisfies
\begin{align}
\label{sec3:eq1}
\e \tanh^2 Y&=c,
\end{align}
where $Y:=w\xi'(c)^{1/2}$. Next, recall a well-known result of Toninelli \cite{Ton02}, which says that above the Almeida-Thouless line, i.e., $\xi''(c)\e\cosh^{-4}Y>1,$ the Parisi measure could not be replica symmetric. Let us notice that although Toninelli's original theorem is dedicated to the SK model, one may find that indeed a similar argument as \cite{Ton02} or Section 13.3 \cite{Tal11} will yield Toninelli's theorem in the mixed even-spin model. Thus, we obtain
\begin{align}
\begin{split}\label{sec3:eq2}
\xi''(c)\e\frac{1}{\cosh^4Y}&\leq 1.
\end{split}
\end{align} 
Now consider the PDE solution $\Phi_{\mu_P}$ corresponding to $\mu_P$ from (\ref{intro:eq14}),
\begin{align*}
\Phi_{\mu_{P}}(x,c)&=\log\cosh x+\frac{1}{2}(\xi'(1)-\xi'(c)).
\end{align*}
Let $\xi_{1,1}=\xi_{2,2}=\xi_{1,2}=\xi,$ $F_1=F_2=\frac{\partial \Phi_{\mu_P}}{\partial x}(\cdot,c)=\tanh$, $h^1=h^2=0$, and $c_1=c_2=c$ in Theorem \ref{sec2:prop3}. From $(\ref{sec3:eq1})$ and $(\ref{sec3:eq2})$, $F_1$ and $F_2$ obviously satisfy $(\ref{sec2:prop3:eq3})$ and $(\ref{sec2:prop3:eq4}).$ Therefore, the function $F$ defined from $(\ref{sec2:prop3:eq1})$ must have a unique solution. However, since $F_1$ and $F_2$ are odd functions, one may see clearly that $0$ and $-c$ are also fixed points of $F$, a contradiction. So $c=0$ and this completes the argument of the case that $\mu_P$ is replica symmetric.
\end{Proof of theorem}

\begin{Proof of theorem} {\bf\ref{sec3:thm1} for replica symmetric breaking $\mu_P$:} Assume now that $\mu_P$ is nontrivial and $c=\min\mbox{supp}\mu_P> 0.$ Note that since $\mu_P$ is not replica symmetric, we can further assume $0<c<1$. Recall $\Phi_{\mu_P}$ from $(\ref{intro:eq16}).$ One would like to expect that similar results as $(\ref{sec3:eq1})$ and $(\ref{sec3:eq2})$ also hold for $\Phi_{\mu_P}$ such that one can apply Theorem \ref{sec2:prop3} to conclude Theorem $\ref{sec3:thm1}.$ It turns out that under the assumption $c>0,$ we have the following,
\begin{align}
\begin{split}\label{sec3:eq3}
\e\left(\frac{\partial\Phi_{\mu_P}}{\partial x}(\chi,c)\right)^2&=c,
\end{split}\\
\begin{split}
\label{sec3:eq4}
\xi''(c)\e\left(\frac{\partial^2\Phi_{\mu_P}}{\partial x^2}(\chi,c)\right)^2&\leq 1,
\end{split}
\end{align}
where $\chi$ denotes the centered Gaussian r.v. with $\e\chi^2=\xi'(c).$ Suppose for the moment that $(\ref{sec3:eq3})$ and $(\ref{sec3:eq4})$ hold (They will be verified below).  From Theorem $\ref{sec2:prop3}$ using $\xi_{1,1}=\xi_{2,2}=\xi_{1,2}=\xi,$ $F_1=F_2=\frac{\partial\Phi_{\mu_P}}{\partial x}(\cdot,c)$, $h^1=h^2=0,$ and $c_1=c_2=c$, the function $F$ defined at $(\ref{sec2:prop3:eq1})$ has a unique fixed point, but this contradicts the fact that $0$ and $-c$ are also the fixed points of $F$ since $\frac{\partial\Phi_{\mu_{P}}}{\partial x}(\cdot,c)$ is odd from $(c)$ in Proposition \ref{add:prop1}. Therefore, $c$ has to be zero, which finishes the proof of Theorem \ref{sec3:thm1}.
\end{Proof of theorem}

\smallskip

The rest of this subsection is devoted to the derivation of $(\ref{sec3:eq3})$ and $(\ref{sec3:eq4})$ assuming $c>0$. The basic idea is to study the local stability of the Parisi solution $\mu_P$ in the Parisi formula as performed in (\ref{sec3:eq1}) and in Chapter 14 \cite{Tal11}. Suppose that $(k,\mathbf{m},\mathbf{q})$ is a triplet corresponding to a measure $\mu.$ Recall $\mathcal{P}_k(\mathbf{m},\mathbf{q})$ from $(\ref{intro:eq5})$. Since we will be differentiating this quantity with respect to $q_p$'s and $m_p$'s and the definition of $X_0$ in $\mathcal{P}_k(\mathbf{m},\mathbf{q})$ involves an iteration scheme, for convenience, we define a sequence of functions $(A_{p})_{0\leq p\leq k+2}$ as follows. Let $(z_p)_{0\leq p\leq k+1}$ be independent centered Gaussian with $\e z_p^2=\xi'(q_{p+1})-\xi'(q_p).$ Starting from $A_{k+2}(x)=\log\cosh x$, we define decreasingly for $0\leq p\leq k+1$,
\begin{align}
\label{add:eq11}
A_p(x)=\frac{1}{m_p}\log \e\exp m_p A_{p+1}(x+z_p),
\end{align}
where we define $A_p(x)=\e A_{p+1}(x+z_p)$ whenever $m_p=0.$ Note that $X_0=\e A_0(h)$. Let $\Phi_{\mu}$ be the PDE solution $(\ref{intro:eq14})$ corresponding to $\mu.$ Easy to see $\Phi_{\mu}(x,1)=A_{k+2}(x)$ and more importantly, a direction computation using Gaussian integration by parts implies that $\Phi_\mu$ can be represented in terms of $(A_p)_{0\leq p\leq k+2},$
\begin{align}
\label{add:eq9}
\Phi_{\mu}(x,q)&=\frac{1}{m_p}\log \e\exp m_pA_{p+1}\left(x+z\sqrt{\xi'(q_{p+1})-\xi'(q)}\right)
\end{align}
whenever $q_p\leq q<q_{p+1}$ for some $0\leq p\leq k+1$, where $z$ is standard Gaussian. In particular, for $0\leq p\leq k+2,$
\begin{align}
\label{add:eq10}
\Phi_{\mu}(x,q_p)=A_p(x).
\end{align}
Set $(\zeta_p)_{1\leq p\leq k+2}$ by letting $$
\zeta_p=\sum_{0\leq n<p}z_p
$$
and set $(W_p)_{1\leq p\leq k+1}$ by $$
W_p=\exp m_{p}(A_{p+1}(\zeta_{p+1})-A_{p}(\zeta_{p})).
$$ 
Now suppose that $k\geq 1$ and $\mu\in\mbox{MIN}(\varepsilon)$ for some $\varepsilon>0$. Let $0\leq s\leq k+1$ satisfy $q_s\leq c<q_{s+1}$. A study of the local stability of $\mu$ in $\mathcal{P}_k$ yields the lemma below.  

\begin{lemma}
\label{sec3:lem2} Suppose that there exists some $0<c'<c$ such that $c'>\varepsilon^{1/6}.$ If $q_s>c'$, then 
\begin{align}
\begin{split}\label{sec3:lem2:eq1}
\e W_{1}\cdots W_{s-1}A_{s}'(\zeta_{s})^2&=q_{s},
\end{split}
\\
\begin{split}\label{sec3:lem2:eq2}
\xi''(q_s)\e W_{1}\cdots W_{s-1}A_{s}''(\zeta_{s})^2&\leq 1+M\varepsilon^{1/6},
\end{split}
\end{align}
where $M>0$ depends only on $\xi$ and $c'.$ 
If $q_{s+1}<1$, then 
\begin{align}
\begin{split}\label{sec3:lem2:eq3}
\e W_{1}\cdots W_{s}A_{s+1}'(\zeta_{s+1})^2&=q_{s+1},
\end{split}
\\
\begin{split}\label{sec3:lem2:eq4}
\xi''(q_{s+1})\e W_{1}\cdots W_{s}A_{s+1}''(\zeta_{s+1})^2&\leq 1+M\varepsilon^{1/6},
\end{split}
\end{align}
where $M>0$ depends only on $\xi$ and $c.$
\end{lemma}

\begin{proof} To obtain $(\ref{sec3:lem2:eq1})$ and $(\ref{sec3:lem2:eq3})$, suppose for the moment that one thinks of $\mathcal{P}_k(\mathbf{m},\mathbf{q})$ as a function defined on the space of all vectors $(m_p)_{0\leq p\leq k+1}$ and $(q_p)_{0\leq p\leq k+2}$ satisfying
\begin{align*}
m_0&=0<m_1<m_2<\cdots<m_{k}<m_{k+1}=1,\\
q_0&=0\leq q_1<q_2<\ldots<q_{k+1}<q_{k+2}=1.
\end{align*}
If $q_1=0,$ a direct differentiation of $\mathcal{P}_k(\mathbf{m},\mathbf{q})$ with respect to $q_r$ for $2\leq r\leq k+1$ implies
\begin{align}\label{sec3:lem2:proof:eq1}
\frac{\partial}{\partial q_r}\mathcal{P}_{k}(\mathbf{m},\mathbf{q})&=\frac{1}{2}(m_r-m_{r-1})\xi''(q_r)(-\e W_1\cdots W_{r-1}A_r'(\zeta_r)^2+q_r);
\end{align}
if $q_1>0,$ then $(\ref{sec3:lem2:proof:eq1})$ also holds for $r=1.$ For the detailed computation, one may refer to Proposition 14.7.5 in \cite{Tal11}. Since $(k,\mathbf{m},\mathbf{q})\in\mbox{MIN}(\varepsilon)$, it implies that $\frac{\partial}{\partial q_r}\mathcal{P}_{k}(\mathbf{m},\mathbf{q})=0$ for either $r\geq 2$ or $r=1$ with $q_1>0.$ Consequently,
\begin{align}\label{sec3:lem2:proof:eq6}
&\e W_1\cdots W_{r-1}A_r'(\zeta_r)^2=q_r
\end{align}
for either $r\geq 2$ or $r=1$ with $q_1>0.$ In particular, if the condition $q_s>c'$ holds, then $1\leq s\leq k+1$ and so $(\ref{sec3:lem2:eq1})$ holds from $(\ref{sec3:lem2:proof:eq6})$ with $r=s$; if $q_{s+1}<1,$ then $1\leq s+1\leq k+1$ and since $q_{s+1}>c$, using $(\ref{sec3:lem2:proof:eq6})$ with $r=s+1$ implies $(\ref{sec3:lem2:eq3}).$

\smallskip

For $(\ref{sec3:lem2:eq2})$ and $(\ref{sec3:lem2:eq4})$, recall that the triplet $(k,\mathbf{m},\mathbf{q})\in\mbox{MIN}(\varepsilon)$ satisfies
\begin{align*}
m_0&=0<m_1<m_2<\cdots<m_k<m_{k+1}=1,\\
q_0&=0\leq q_1<q_2<\cdots<q_{k+1}<q_{k+2}=1.
\end{align*}
Consider new lists of sequences, for $1\leq r\leq k+1,$
\begin{align*}
\mathbf{m}(m)&=(0,m_1,\ldots,m_{r-1},m,m_r,\ldots,m_k,1),\\
\mathbf{q}(u)&=(0,q_1,\ldots,q_{r-1},u,q_r,\ldots,q_{k+1},1),
\end{align*} 
with $m_{r-1}\leq m\leq m_r$ and $q_{r-1}\leq u\leq q_r.$ For $1\leq r\leq k+1,$ we define
$$
f_r(u)=\left.\frac{\partial}{\partial m}\mathcal{P}_{k+1}(\mathbf{m}(m),\mathbf{q}(u))\right|_{m=m_{r-1}}
$$
Let us observe that from the definition of $f$ and $(k,\mathbf{m},\mathbf{q})\in\mbox{MIN}(\varepsilon),$
\begin{align}
\begin{split}\label{sec3:lem2:proof:eq9}
f_r(q_r)&=0,\,\forall 1\leq r\leq k+1,
\end{split}
\\
\begin{split}\label{sec3:lem2:proof:eq10}
f_{r}(q_{r-1})&=0,\,\forall 2\leq r\leq k+1.
\end{split}
\end{align}
Here comes the most critical part: there exists a constant $M>0$ depending only on $\xi$ such that for every $2\leq r\leq k+1,$
\begin{align}
\begin{split}
\label{sec3:lem2:proof:eq2}
f_r(u)\geq -M\sqrt{\varepsilon},\,\forall q_{r-1}\leq u\leq q_r,
\end{split}\\
\begin{split}
\label{sec3:lem2:proof:eq3}
&f_r'(q_r)=-\frac{1}{2}\xi''(q_r)\left(\e W_1\cdots W_{r-1}A_r'(\zeta_r)^2-q_r\right),
\end{split}\\
\begin{split}
\label{sec3:lem2:proof:eq4}
&f_r''(q_r)=-\frac{1}{2}\xi''(q_r)\left(\xi''(q_r)\e W_1\cdots W_{r-1}A_r''(\zeta_r)^2-1\right)
\end{split}
\end{align}
and if $q_1>0,$ these also hold for $r=1.$ In addition, for $1\leq r\leq k+1,$
\begin{align}
\begin{split}
\label{sec3:lem2:proof:eq5}
&\max_{q_{r-1}\leq m\leq q_r}|f_r'''(u)|\leq M.
\end{split}
\end{align}
The inequality $(\ref{sec3:lem2:proof:eq2})$ is mainly due to $(k,\mathbf{m},\mathbf{q})\in\mbox{MIN}(\varepsilon),$ while $(\ref{sec3:lem2:proof:eq3}),$ $(\ref{sec3:lem2:proof:eq4}),$ and $(\ref{sec3:lem2:proof:eq5})$ are based on a series of applications of the Gaussian integration by parts formula. Again, since they have been carried out in great detail in Section 14.7 \cite{Tal11}, we will omit the derivation of these results. Now, using Taylor's formula together with $(\ref{sec3:lem2:proof:eq6}),$ $(\ref{sec3:lem2:proof:eq9})$, $(\ref{sec3:lem2:proof:eq3})$, and $(\ref{sec3:lem2:proof:eq5})$, if either $r\geq 2$ or $r=1$ with $q_1>0,$ we have
\begin{align}\label{sec3:lem2:proof:eq7}
f_r(u)&\leq \frac{1}{2}(u-q_r)^2f_r''(q_r)+M|u-q_r|^3.
\end{align}
Suppose that $q_s>c'.$ Then $1\leq s\leq k+1$. First, we assume 
\begin{align}\label{sec3:lem2:proof:eq8}
u=q_s-\varepsilon^{1/6}\geq q_{s-1}.
\end{align}
Using $(\ref{sec3:lem2:proof:eq2})$ and $(\ref{sec3:lem2:proof:eq7})$ with $r=s$ yields $$
-M\sqrt{\varepsilon}\leq \frac{1}{2}\varepsilon^{1/3}f_s''(q_s)+M\sqrt{\varepsilon}
$$
and this implies from $(\ref{sec3:lem2:proof:eq4})$,
$$
-f_s''(q_s)=\frac{1}{2}\xi''(q_s)(\xi''(q_s)\e W_1\cdots W_{s-1} A_s''(\zeta_s)^2-1)\leq 4M\varepsilon^{1/6}.
$$
Since $\xi''(c')<\xi''(q_s),$ $(\ref{sec3:lem2:eq2})$ clearly follows. Assume now that $(\ref{sec3:lem2:proof:eq8})$ fails. Since $c'>\varepsilon^{1/6}$ and $q_0=0,$ we have $s\geq 2.$ Therefore, the use of $(\ref{sec3:lem2:proof:eq10})$, $(\ref{sec3:lem2:proof:eq4}),$ and $(\ref{sec3:lem2:proof:eq7})$ with $r=s$ and $u=q_{s-1}$ leads to
$$
-f_s''(q_s)=\frac{1}{2}\xi''(q_s)(\xi''(q_s)\e W_1\cdots W_{s-1} A_s''(\zeta_s)^2-1)\leq 2M(q_s-q_{s-1})\leq 2M\varepsilon^{1/6}.
$$
Again, $(\ref{sec3:lem2:eq2})$ holds from this inequality and using $\xi''(q_s)>\xi''(c').$ Note that $q_{s+1}>c>\varepsilon^{1/6}$ and that $q_{s+1}<1$ implies $1\leq s+1\leq k+1$. One may argue similarly as above to get $(\ref{sec3:lem2:eq4}).$
\end{proof}

\begin{lemma}
\label{sec3:lem3}
Let $\eta>0$ and $0<\delta<c$. Suppose that $l$ and $l'$ are fixed integers with $1\leq l<l'\leq k+1.$ If $m_p\leq \eta$ for every $1\leq p\leq l-1$, then
\begin{align}
\label{sec3:lem3:eq1}
\e|W_1W_2\cdots W_{l-1}-1|\leq M\eta.
\end{align}
If $c-\delta\leq q_p\leq q_{l'}$ for every $l\leq p\leq l',$ then
\begin{align}
\label{sec3:lem3:eq2}
\e W_1W_2\cdots W_{l-1}|W_lW_{l+1}\cdots W_{l'-1}-1|\leq M\sqrt{q_{l'}-c+\delta}.
\end{align}
Here, $M$ depends only on $\xi.$
\end{lemma}

\begin{proof}
Similar arguments as $(14.468)$ and $(14.469)$ in \cite{Tal11} will yield the announced results. 
\end{proof}

Recall the definition of the Parisi measure $\mu_P$. It is the weak limit of a sequence of probability measures $\mu_n\in\mbox{MIN}(\varepsilon_n)$ with $\varepsilon_n\downarrow 0.$ For clarity, in the following, we will only use $(k,\mathbf{m},\mathbf{q})$ to denote the triplet corresponding to $\mu_n.$ One has to keep in mind that this triplet depends on $n$ and $\varepsilon_n.$ Note that since $\mu_P$ is nontrivial, we may assume $k\geq 1$ for all $n\geq 1.$ Let $0\leq s\leq k+1$ satisfy $q_{s}\leq c<q_{s+1}.$ Without loss of generality, we may assume that the limits of $q_{s}$, $q_{s+1},$ and $m_{s}$ exist and they are denoted by $c_-$, $c_+,$ and $m_c$, respectively. Note that if $c_-<c<c_+,$ then $c_-<c$ implies $m_c=0,$ but on the other hand, $c<c_+$ implies $\min\mbox{supp}\mu_P>c$, a contradiction. Therefore, we can further assume that one of the following occurs.
\begin{itemize}
\item[$(i)$] $c_-=c$ and there is some $0<c'<c$ such that $q_{s}>c'$ for all $n$.
\item[$(ii)$] $c_+=c$ and $q_{s+1}<1$ for all $n.$
\end{itemize}

\begin{lemma}
\label{sec3:lem4} We have that
\begin{align}\label{sec3:lem4:eq1}
\lim_{n\rightarrow\infty}\e|W_1\cdots W_{s-1}-1|=0.
\end{align}
If in addition, $(ii)$ occurs, then we also have 
\begin{align}
\label{sec3:lem4:eq2}
\lim_{n\rightarrow\infty}|W_1\cdots W_s-1|=0.
\end{align}
\end{lemma}

\begin{proof}
Let $0<\delta<c$ be fixed. Suppose that $1\leq l\leq s+1$ is the largest integer such that $q_{l-1}\leq c-\delta.$ Since $\lim_{n\rightarrow\infty}\mu_n([0,c-\delta])=0$, we have that for any $\eta>0$, $m_p\leq \eta$ for every $0\leq l-1$ provided that $n$ is sufficiently large. Since $c-\delta<q_p\leq q_s\leq c<q_{s+1}$ for $l\leq p\leq s,$ using $(\ref{sec3:lem3:eq2})$ twice, we get
\begin{align}
\begin{split}\label{sec3:lem4:proof:eq2}
\e W_1W_2\cdots W_{l-1}|W_lW_{l+1}\cdots W_{s-1}-1|\leq M\sqrt{q_s-c+\delta}\leq M\sqrt{\delta},
\end{split}\\
\begin{split}\label{sec3:lem4:proof:eq3}
\e W_1W_2\cdots W_{l-1}|W_lW_{l+1}\cdots W_s-1|\leq M\sqrt{q_{s+1}-c+\delta}.
\end{split}
\end{align}
From the triangle inequality, $(\ref{sec3:lem3:eq1})$, and $(\ref{sec3:lem4:proof:eq2}),$ we have that
\begin{align*}
\limsup_{n\rightarrow \infty}\e|W_1W_2\cdots W_{s-1}-1|
&\leq \limsup_{n\rightarrow\infty}\e W_1W_2\cdots W_{l-1}|W_{l}W_{l+1}\cdots W_{s-1}-1|\\
&+\limsup_{n\rightarrow\infty}\e|W_1W_2\cdots W_{l-1}-1|\\
&\leq M\sqrt{\delta}+M\eta.
\end{align*}
Similarly, if $(ii)$ occurs, using the triangle inequality, $(\ref{sec3:lem3:eq1})$, and $(\ref{sec3:lem4:proof:eq3}),$ we obtain
\begin{align*}
\limsup_{n\rightarrow\infty}\e|W_1W_2\cdots W_s-1|&\leq \limsup_{n\rightarrow\infty}\e W_1W_2\cdots W_{l-1}|W_lW_{l+1}\cdots W_s-1|\\
&+\limsup_{n\rightarrow\infty}\e|W_1W_2\cdots W_{l-1}-1|\\
&\leq \lim_{n\rightarrow\infty}M\sqrt{q_{s+1}-c+\delta}+M\eta\\
&=M\sqrt{\delta}+M\eta.
\end{align*}
Since $\delta,\eta>0$ are arbitrary small numbers, passing to the limit implies $(\ref{sec3:lem4:eq1})$ and $(\ref{sec3:lem4:eq2})$.
\end{proof}

Now let us proceed to prove $(\ref{sec3:eq3})$ and $(\ref{sec3:eq4})$ as follows. Suppose that $(i)$ holds. Then from $(\ref{add:eq10}),$ $(a)$ in Proposition $\ref{add:prop1}$,  $(\ref{sec3:lem2:eq1})$, $(\ref{sec3:lem2:eq2})$, and $(\ref{sec3:lem4:eq1}),$ we have
\begin{align*}
\e\left(\frac{\partial\Phi_{\mu_P}}{\partial x}(\chi,c)\right)^2&=\lim_{n\rightarrow\infty}\e A_s'(\zeta_s)^2\\
&=\lim_{n\rightarrow\infty}\e W_1\cdots W_{s-1}A_s'(\zeta_s)^2\\
&=\lim_{n\rightarrow \infty}q_s\\
&=c
\end{align*}
and 
\begin{align*}
\xi''(c)\e\left(\frac{\partial^2\Phi_{\mu_P}}{\partial x^2}(\chi,c)\right)^2&=\lim_{n\rightarrow\infty}\xi''(q_s)\e A_s''(\zeta_s)^2\\
&= \lim_{n\rightarrow\infty}\xi''(q_s)\e W_1\cdots W_{s-1}A_s''(\zeta_s)^2\\
&\leq 1.
\end{align*}
If $(ii)$ holds, then we argue similarly by using from $(\ref{add:eq10}),$ $(a)$ in Proposition $\ref{add:prop1}$,  $(\ref{sec3:lem2:eq3})$, $(\ref{sec3:lem2:eq4})$, and $(\ref{sec3:lem4:eq2})$ to conclude $(\ref{sec3:eq3})$ and $(\ref{sec3:eq4}).$ This completes the argument of our proof.

\section{Controlling the coupled free energy}\label{sec1}

We will recall Guerra's replica symmetry breaking bound for the coupled free energy $(\ref{sec0:eq1})$. From this, we derive a manageable bound by using suitable chosen parameters. As one shall see, this derivation naturally gives rise to a crucial function that will be used in Section $\ref{sec2}$ to determine the unique constant $u_f$ as stated in our chaos results and also to control the behavior of the overlap $R(\vsi,\vtau)$ as $(\ref{add5}).$

\subsection{Guerra's bound}
Recall the two systems corresponding to the Hamiltonians $H_N^1$ and $H_N^2$ in $(\ref{intro:eq18})$. We denote by $Z_N^1$ and $Z_N^2$ the partition functions, by $\mathcal{P}_{k_1}^1(\mathbf{m}^1,\mathbf{q}^1)$ and $\mathcal{P}_{k_2}^2(\mathbf{m}^2,\mathbf{q}^2)$ as in $(\ref{intro:eq5}),$ and by  $\mathcal{P}^1(\xi_{1,1},h^1)$ and $\mathcal{P}^2(\xi_{2,2},h^2)$ the variational formulas as in $(\ref{intro:eq13})$ associated to the two systems, respectively. Set $u_{1,1}=u_{2,2}=1$ and $u_{1,2}=u_{2,1}=u$ for some $-1\leq u\leq 1.$ Recall $\xi_{1,1}$, $\xi_{2,2},$ and $\xi_{1,2}$ from $(\ref{add:eq2})$. Define $\xi_{2,1}=\xi_{1,2}$ and $\theta_{j,j'}(x)=x\xi_{j,j'}'(x)-\xi_{j,j'}(x)$ for $1\leq j,j'\leq 2.$ Let $\kappa\geq 1$ be an integer and let $(y_p^1,y_p^2)$ be jointly centered Gaussian r.v. for $0\leq p\leq \kappa$ with
\begin{align}\label{Guerra:eq1}
\e y_p^j y_p^{j'}=\xi_{j,j'}'(\rho_{p+1}^{j,j'})-\xi_{j,j'}'(\rho_{p}^{j,j'}),
\end{align}
where $(\rho_p^{j,j'})_{0\leq p\leq \kappa+1,1\leq j,j'\leq 2}$ are real numbers satisfying $\rho_0^{j,j'}=0$,  $\rho_{\kappa+1}^{j,j'}=u_{j,j'}$ for $1\leq j,j'\leq 2.$ These pairs $(y_p^1,y_p^2)$ are also assumed to be independent of each other. Let $n_0=0\leq n_1\leq \cdots\leq n_{\kappa-1}\leq n_\kappa=1.$ Recall the coupled free energy $p_{N,u}$ from \eqref{sec0:eq1}. The Guerra replica symmetry breaking bound for the coupled free energy is stated as follows.

\begin{theorem}[Guerra]\label{Guerra} We have
\begin{align}\label{Guerra:thm:eq1}
p_{N,u}&\leq 2\log 2+Y_0(\lambda)-\lambda u-\frac{1}{2}\sum_{j,j'\leq 2}\sum_{0\leq p\leq \kappa}n_p(\theta_{j,j'}(\rho_{p+1}^{j,j'})-\theta_{j,j'}(\rho_{p}^{j,j'})),
\end{align}
where $Y_0(\lambda)$ is defined as follows. Starting with 
\begin{align*}
Y_{\kappa+1}(\lambda)&:=\log\left(\cosh\left(h^1+\sum_{0\leq p\leq \kappa}y_p^1\right)\cosh\left(h^2+\sum_{0\leq p\leq \kappa}y_p^2\right)\cosh\lambda\right.\\
&\left.\qquad\qquad+\sinh\left(h^1+\sum_{0\leq p\leq \kappa}y_p^1\right)\sinh\left(h^2+\sum_{0\leq p\leq \kappa}y_p^2\right)\sinh\lambda\right),
\end{align*}
we define decreasingly for $p\geq 1,$ $Y_p(\lambda)=n_p^{-1}\log \e_p\exp n_pY_{p+1}(\lambda)$, where $\e_p$ denotes the expectation in the r.v. $y_n^1$ and $y_n^2$ for $n\geq p.$ In the case of $n_p=0$ for some $p$, we set $Y_p(\lambda)=\e_p Y_{p+1}(\lambda).$ Finally, $Y_0(\lambda)=\e Y_1(\lambda).$
\end{theorem}

Recalling Guerra's original bound from $(\ref{intro:eq6})$, $(\ref{Guerra:thm:eq1})$ is a kind of two dimensional bound for the coupled free energy. Its proof is essentially the same as that of Proposition 14.12.4 \cite{Tal11} and a more generalized version can be found in Section 15.7 \cite{Tal11}. 
Such bound has played a very fundamental role in Talagrand's original proof for the validity of the Parisi formula \cite{Tal06}, where the two systems he considered are exactly the same, i.e., $\vec{\beta}_1=\vec{\beta}_2,$ $\mathcal{G}^1=\mathcal{G}^2$, and $h^1=h^2.$ In our case, since these external parameters may be essentially different, how to find suitable parameters $\kappa,$ $(n_p)_{0\leq p\leq \kappa+1}$, $(\rho_{p}^{j,j'})_{0\leq p\leq \kappa+1, 1\leq j,j'\leq 2},$ and $\lambda$ to control this bound becomes a very intricate issue. To illustrate the main difficulty, note that from the definition $(\ref{sec0:eq1})$ of $p_{N,u}$ and $(\ref{intro:eq6})$, one sees obviously for all $u\in S_N,$
\begin{align}
\label{add:eq7}
p_{N,u}\leq \frac{1}{N}\e\log Z_N^1+\frac{1}{N}\e\log Z_N^2\leq \mathcal{P}_{k_1}^1(\mathbf{m}^1,\mathbf{q}^1)+\mathcal{P}_{k_2}^2(\mathbf{m}^2,\mathbf{q}^2)
\end{align}
for arbitrary choices of the triplets $(k_1,\mathbf{m}^1,\mathbf{q}^1)$ and $(k_2,\mathbf{m}^2,\mathbf{q}^2)$ satisfying $(\ref{intro:eq0}).$ Thus, if $(\ref{Guerra:thm:eq1})$ is a relevant bound to investigate chaos problems, one should be able to find parameters for $(\ref{Guerra:thm:eq1})$ to recover the inequality $(\ref{add:eq7})$. In the next subsection, one shall see that this can be done for $|u|\leq \sqrt{c_1c_2},$ but the general case remains unclear.

\subsection{a manageable bound}\label{amb}

The goal of this subsection is to derive the following bound for the coupled free energy. Let $\mathbf{m}=(m_p)_{0\leq p\leq k+1}$ satisfy \eqref{intro:eq0}. Consider two triplets $(k,\mathbf{m},\mathbf{q}^1)$ and $(k,\mathbf{m},\mathbf{q}^2)$. We denote by $\mu^1$ and $\mu^2$ the probability measures induced by these two triplets and by $\Phi_{1,\mu^1}$ and $\Phi_{2,\mu^2}$ the PDE solutions $(\ref{intro:eq14})$ associated with $\xi_{1,1},$ $\mu^1$ and $\xi_{2,2}$, $\mu^2,$ respectively. 

\begin{proposition}
\label{sec1:prop} Let $0< v_1,v_2<1.$ Suppose that $v_1=q_{\iota}^1$ and $v_2=q_{\iota}^2$ for some $1\leq \iota\leq k+1$. For every $|u|\leq \sqrt{v_1v_2},$ we have that
\begin{align}
\begin{split}\label{sec1:lem1:eq1}
p_{N,u}&\leq \mathcal{P}_{k}^1(\mathbf{m},\mathbf{q}^1)+\mathcal{P}_{k}^2(\mathbf{m},\mathbf{q}^2)\\
&-\frac{1}{2}\left(\e \frac{\partial\Phi_{1,\mu^1}}{\partial x}(h^1+\chi^1,v_1)\frac{\partial\Phi_{2,\mu^2}}{\partial x}(h^2+\chi^2,v_2)-u\right)^2\\
&+\frac{1}{2}\sum_{p=0}^{{\iota}-1}m_p(\theta_{1,1}(q_{p+1}^1)-\theta_{1,1}(q_p^1))
+\frac{1}{2}\sum_{p=0}^{{\iota}-1}m_p(\theta_{2,2}(q_{p+1}^2)-\theta_{2,2}(q_p^2)),
\end{split}
\end{align}
where $\chi^1$ and $\chi^2$ are jointly centered Gaussian independent of $h^1$ and $h^2$ with $\e (\chi^1)^2=\xi_{1,1}'(v_1)$, $\e(\chi^2)^2=\xi_{2,2}'(v_2)$, and $\e\chi^1\chi^2=\xi_{1,2}'(u)$.
\end{proposition}

We will need a crucial lemma. Let us keep every parameter but $\lambda$ in the statement of Theorem $\ref{Guerra}$ fixed. Recall $\kappa,$ $(n_p)_{0\leq p\leq \kappa}$, and $(y_p^j)_{0\leq p\leq \kappa,1\leq j\leq 2}$ from the last subsection. Starting with $D_{j,\kappa+1}(x)=\log\cosh x$ for $j=1,2$, we define decreasingly for $1\leq p\leq \kappa$ and $j=1,2$ by 
$$
D_{j,p}(x)=\frac{1}{n_p}\log \e_p\exp n_p D_{j,p+1}(x+y_p^j).
$$  
As usual, we define $D_{j,p}(x)=\e_pD_{j,p+1}(x+y_p^j)$ when $n_p=0.$ 

\begin{lemma}\label{lem1}
Suppose that $(y_p^1)_{1\leq p\leq \kappa}$ and $(y_p^2)_{1\leq p\leq \kappa}$ are independent of each other. Then
\begin{align}
\begin{split}\label{lem1:eq1}
Y_0(0)&=\e D_{1,1}(h^1+y_0^1)+\e D_{2,1}(h^2+y_0^2),
\end{split}
\\
\begin{split}
\label{lem1:eq2}
Y_0'(0)&=\e D_{1,1}'(h^1+y_0^1)D_{2,1}'(h^2+y_0^2).
\end{split}
\end{align}
For the second derivative of $Y_0,$ we have for every $\lambda,$
\begin{align}
\label{lem1:eq3}
0\leq Y_0''(\lambda)\leq 1.
\end{align}
\end{lemma}

\begin{proof} The proofs for $(\ref{lem1:eq1})$ and $(\ref{lem1:eq2})$ are exactly the same as the arguments in Proposition 14.6.4 \cite{Tal11}, while the statement $(\ref{lem1:eq3})$ can also be obtained from a similar argument as Lemma 14.6.5 \cite{Tal11}.
\end{proof}

Similar to $(\ref{add:eq11})$, we define two sequences of functions $(A_{1,p})_{0\leq p\leq k_1+2}$ and $(A_{2,p})_{0\leq p\leq k_2+2}$ as follows. For $j=1,2,$ suppose that $(z_p^j)_{0\leq p\leq k_j+1}$ are independent centered Gaussian r.v. with $\e (z_p^j)^2=\xi_{j,j}'(q_{p+1}^j)-\xi_{j,j}'(q_p^j).$ Starting with $A_{j,k_j+2}(x)=\log \cosh x,$ we define decreasingly
\begin{align*}
A_{j,p}(x)&=\frac{1}{m_p}\log \e\exp m_p A_{j,p+1}(x+z_p^j)
\end{align*}
for $0\leq p\leq k_j+1$, where we let $A_{j,p}(x)=\e A_{j,p+1}(x+z_p^j)$ when $m_p=0.$  

\smallskip
\smallskip

\begin{Proof of proposition} {\bf \ref{sec1:prop}:} Let us specify the parameters $\kappa,$ $(n_p)_{0\leq p\leq \kappa}$, $(\rho_{p}^{j,j'})_{0\leq p\leq\kappa+1,1\leq j,j'\leq 2}$, and $\lambda$ in Guerra's bound as follows:
\begin{align*}
\kappa&=k-\iota+2,\\
n_0&=0,n_1=m_\iota,\ldots,n_{\kappa}=m_{k+1},\\
\rho_{0}^{1,1}&=0,\rho_{1}^{1,1}=q_\iota^1,\ldots,\rho_{\kappa+1}^{1,1}=q_{k+2}^1,\\
\rho_{0}^{2,2}&=0,\rho_{1}^{2,2}=q_\iota^2,\ldots,\rho_{\kappa+1}^{2,2}=q_{k+2}^2,\\
\rho_{0}^{1,2}&=0,\rho_{1}^{1,2}=u,\ldots,\rho_{\kappa+1}^{1,2}=u,\\
\rho_{0}^{2,1}&=0,\rho_{1}^{2,1}=u,\ldots,\rho_{\kappa+1}^{2,1}=u.
\end{align*}
Recall the jointly centered Gaussian r.v. $(y_p^1,y_p^2)$ for $0\leq p\leq \kappa$ defined in the statement of Guerra's inequality. 
Let us emphasize that with this special choice $(\rho_{p}^{j,j'})_{0\leq p\leq \kappa+1,1\leq j,j'\leq 2}$ and the assumption $|u|\leq \sqrt{q_\iota^1q_\iota^2}=\sqrt{v_1v_2},$ the existence of $(y_p^1,y_p^2)$ for $0\leq p\leq \kappa$ is ensured by \eqref{sec2.2:eq1}. We use these parameters for $Y_0(\lambda)$. Applying the mean value theorem and $(\ref{lem1:eq3})$ to $Y_0(\lambda)-\lambda u$ gives $$
Y_0(\lambda)-\lambda u\leq Y_0(0)+(Y_0'(0)-u)\lambda+\frac{\lambda^2}{2},\,\forall\lambda.
$$
Minimizing the right-hand side of this inequality with respect to $\lambda$ and using \eqref{Guerra:thm:eq1} yield
\begin{align}
\begin{split}\label{add:eq12}
p_{N,u}&\leq 2\log 2+ Y_0(0)-\frac{1}{2}(Y_0'(0)-u)^2-\frac{1}{2}\sum_{j\leq 2}\sum_{p=\iota}^{k+1}m_p(\theta_{j,j}(q_{p+1}^j)-\theta_{j,j}(q_p^j)).
\end{split}
\end{align}
To complete the proof, it remains to check that the three terms on the right-hand side of $(\ref{add:eq12})$ together give the asserted inequality. Note that from \eqref{Guerra:eq1} and our construction, $(y_p^1)_{1\leq p\leq \kappa}$ is independent of $(y_p^2)_{1\leq p\leq \kappa}.$ So \eqref{lem1:eq1} and \eqref{lem1:eq2} holds. Using the definitions of $A_{j,p}$'s and $D_{j,p}$'s, one sees that 
\begin{align}
\label{add:eq13}
D_{1,1}=A_{1,\iota}\,\,\mbox{and}\,\, D_{2,1}=A_{2,\iota}.
\end{align} 
From Jensen's inequality, $$
A_{j,p}(x)=\frac{1}{m_p^j}\log\e\exp m_p^j A_{j,p+1}(x+z_p^j)\geq \e A_{j,p+1}(x+z_p^j)
$$
and by decreasing induction on $p,$ 
$$\e A_{j,\iota}\biggl(h^j+\sum_{0\leq p<\iota}z_p^j\biggr)\leq \e A_{j,0}(h^j)=X_0^j,$$
where $X_0^j$ is defined as in $(\ref{intro:eq5})$ using $(k,\mathbf{m},\mathbf{q}^j)$, $\xi_{j,j},$ and $h^j.$ Since $y_0^j$ is equal to $\sum_{0\leq p<\iota}z_p^j$ in distribution, it follows from the last inequality, $(\ref{lem1:eq1}),$ and $(\ref{add:eq13})$ that
\begin{align}
\label{sec1:lem1:eq5}
Y_0(0)&\leq X_0^1+X_0^2.
\end{align} 
Next, we compute $Y_0'(0).$ Similar to $(\ref{add:eq9})$ and $(\ref{add:eq10})$, the function $A_{j,\iota}$ and $\Phi_{j,\mu^j}$ are related by
$A_{j,\iota}(x)=\Phi_{j,\mu^j}(x,q_\iota^j)=\Phi_{j,\mu^j}(x,v_j).$ From this, $(\ref{lem1:eq2})$, and $(\ref{add:eq13}),$ we have
\begin{align}
\label{add:eq14}
Y_0'(0)&=\e\frac{\partial\Phi_{1,\mu^1}}{\partial x}(h^1+y_0^1,v_1)\frac{\partial\Phi_{2,\mu^2}}{\partial x}(h^2+y_0^2,v_2).
\end{align}
Finally, rewrite
\begin{align*}
&\sum_{j\leq 2}\sum_{p=\iota}^{k+1}m_p(\theta_{j,j}(q_{p+1}^j)-\theta_{j,j}(q_p^j))\\
&=\sum_{j\leq 2}\sum_{p=0}^{k+1}m_p(\theta_{j,j}(q_{p+1}^j)-\theta_{j,j}(q_p^j))
-\sum_{j\leq 2}\sum_{p=0}^{\iota-1}m_p(\theta_{j,j}(q_{p+1}^j)-\theta_{j,j}(q_p^j)).
\end{align*}
Combining $(\ref{add:eq12})$, $(\ref{sec1:lem1:eq5}),$ $(\ref{add:eq14}),$ and this equation together completes our proof.
\end{Proof of proposition}

\section{Determination of the location for the overlap}\label{sec2}

Recall the Gibbs measures $G_N^1$ and $G_N^2$ using the Hamiltonians $H_N^1$ and $H_N^2$ from $(\ref{intro:eq18})$ for the two mixed even-spin systems introduced in Section \ref{sec:intro}. Throughout this section, we will assume that their external fields $h^1$ and $h^2$ are sub-Gausian with $\e(h^1)^2\neq 0$ and $\e(h^2)^2\neq 0$ and $h^1$ and $h^2$ might not be independent. For $j=1,2,$ let us denote by $\mu_{P}^j$ a Parisi measure for the $j$-th system and by $\Phi_{j,\mu_{P}^j}$ the function defined as in $(\ref{intro:eq16})$ using $\xi_{j,j},h^j,\mu_{P}^j$. From the positivity of the Parisi measure (\ref{intro:eq20}), $c_j=\min\mbox{supp}\mu_P^j>0.$ Let us also remark that $c_j<1$. Indeed, as in $(\ref{sec3:eq1})$, if $c_j=1,$ it means that $\mu_P^j$ is replica symmetric and $c_j$ has to satisfy $c_j=\e\tanh^2 Y<1,$ where $Y=\xi_{j,j}'(c_j)^{1/2}w$ for some standard Gaussian $w.$ We will use the bound (\ref{sec1:lem1:eq1}) to determine the constant $u_f$ stated in Theorems $\ref{thm:tc},$ $\ref{thm:dc},$ and $\ref{thm:efc}$ and furthermore, to study the behavior of the overlap $R(\vsi,\vtau)$ in $[-\sqrt{c_1c_2},\sqrt{c_1c_2}].$

\subsection{Determination of $u_f$}

Suppose that $0\leq v_1,v_2\leq 1.$ We define  
\begin{align}\label{chaos:eq1}
\phi_{v_1,v_2}(u)&=\e\frac{\partial\Phi_{1,\mu_{P}^1}}{\partial x}(h^1+\chi^1,v_1)\frac{\partial\Phi_{2,\mu_{P}^2}}{\partial x}(h^2+\chi^2,v_2)
\end{align}
for $|u|\leq \sqrt{v_1v_2},$ where $\chi^1$ and $\chi^2$ are jointly centered Gaussian with $\e(\chi^1)^2=\xi_{1,1}'(v_1),$ $\e(\chi^2)^2=\xi_{2,2}'(v_2),$ and $\e\chi^1\chi^2=\xi_{1,2}'(u).$ As we have explained in the discussion right before Theorem \ref{sec2:prop3}, the existence of $(\chi^1,\chi^2)$ is guaranteed and thus $\phi_{v_1,v_2}$ is well-defined. The constant $u_f$ can be determined by $\phi_{c_1,c_2}$ through the following proposition. 

\begin{proposition}\label{intro:prop1} 
$\phi_{c_1,c_2}$ maps $[-\sqrt{c_1c_2},\sqrt{c_1c_2}]$ into itself and it has a unique fixed point $u_f$. In addition, if $h^1$ and $h^2$ are independent and symmetric with respect to the origin, then $u_f=0.$
\end{proposition}

\begin{proof} Since $\e(h^1)^2\neq 0$ and $\e(h^2)^2\neq 0,$ from $(d)$ in Proposition $\ref{add:prop1}$, we have that
\begin{align}
\begin{split}\label{sec2:prop1:eq1}
\e\frac{\partial\Phi_{j,\mu_{P}^j}}{\partial x}(h^j+\chi^j,c_j)^2&=c_j,
\end{split}
\\
\begin{split}\label{sec2:prop1:eq2}
\xi_{j,j}''(c_j)\e\frac{\partial^2\Phi_{j,\mu_{P}^j}}{\partial x^2}(h^j+\chi^j,c_j)^2&\leq 1.
\end{split}
\end{align}
Let $v_1=c_2$, $v_2=c_2,$ and set 
$$
F_1(x)=\frac{\partial\Phi_{1,\mu_P^1}}{\partial x}(x,c_1)\,\,\mbox{and}\,\,F_2(x)=\frac{\partial\Phi_{2,\mu_{P}^2}}{\partial x}(x,c_2).
$$
Then $(\ref{sec2:prop1:eq1})$ and $(\ref{sec2:prop1:eq2})$ imply $(\ref{sec2:prop3:eq3})$ and $(\ref{sec2:prop3:eq4}).$ In addition, from $(b)$ in Proposition \ref{add:prop1}, $F_1$ and $F_2$ satisfy the required assumptions of Theorem \ref{sec2:prop3}. Thus, $\phi_{c_1,c_2}$ maps $[-\sqrt{c_1c_2},\sqrt{c_1c_2}]$ into itself and has a unique fixed point $u_f.$ Suppose that $h^1$ and $h^2$ are independent and symmetric with respect to the origin. From $(c)$ in Proposition \ref{add:prop1}, since $$
\frac{\partial\Phi_{1,\mu_P^1}}{\partial x}(\cdot,c_1)\,\,\mbox{and}\,\,\frac{\partial\Phi_{2,\mu_P^2}}{\partial x}(\cdot,c_2)
$$ 
are odd functions, one may see clearly that 
$$
\phi_{c_1,c_2}(0)=\e\frac{\partial\Phi_{1,\mu_P^1}}{\partial x}(h^1+\chi^1,c_1)\cdot\e\frac{\partial\Phi_{2,\mu_P^2}}{\partial x}(h^2+\chi^2,c_1)=0\cdot 0=0,
$$
where in this case $\chi^1$ and $\chi^2$ are independent. Thus, $u_f=0$ and this completes our proof.
\end{proof}

\subsection{the behavior of the overlap in $[-\sqrt{c_1c_2},\sqrt{c_1c_2}]$ }

Recall that $\left<\cdot\right>$ is the Gibbs average with respect to $(G_N^1\times G_N^2)^{\otimes\infty}.$ The behavior of the overlap $R(\vsi,\vtau)$ inside $[-\sqrt{c_1c_2},\sqrt{c_1c_2}]$ can be described by the theorem below. For $c_1\leq v_1<1$, $c_2\leq v_2<1$, and $\varepsilon>0,$ we define a set
\begin{align*}
S_\varepsilon(u_f,v_1,v_2)=\left\{
\begin{array}{ll}
\{x:-\sqrt{v_1v_2}\leq x\leq \sqrt{c_1c_2},|x-u_f|\geq \varepsilon\},&\mbox{if $u_f=\sqrt{c_1c_2}$},\\
\\
\{x:-\sqrt{c_1c_2}\leq x\leq \sqrt{v_1v_2},|x-u_f|\geq \varepsilon\},&\mbox{if $u_f=-\sqrt{c_1c_2}$},\\
\\
\{x:-\sqrt{v_1v_2}\leq x\leq \sqrt{v_1v_2},|x-u_f|\geq \varepsilon\},&\mbox{if $|u_f|<\sqrt{c_1c_2}$}.
\end{array}
\right.
\end{align*}

\begin{theorem}\label{sec2:thm1} Let $u_f$ be the fixed point of $\phi_{c_1,c_2}.$ For any $\varepsilon>0$, there exist $c_1<v_1<1$, $c_2<v_2<1$, and $K>0$ that are all independent of $N$ such that for $N\geq 1,$
\begin{align}\label{sec2:thm1:eq1}
\e\left<I\left(R(\vsi,\vtau)\in S_\varepsilon(u_f,v_1,v_2)\right)\right>\leq K\exp\left(-\frac{N}{K}\right).
\end{align}
\end{theorem}

The core of the proof for this theorem is based on the following proposition. Similar to (\ref{intro:eq13}), let $\mathcal{P}^1(\xi_{1,1},h^1)$ and $\mathcal{P}^2(\xi_{2,2},h^2)$ be the variational formulas corresponding to the two systems.

\begin{proposition}\label{sec2:prop2}
For any two $0<v_1,v_2<1$, we have that for $|u|\leq \sqrt{v_1v_2}$,
\begin{align}
\begin{split}
\label{sec2:prop2:eq1}
p_{N,u}&\leq \mathcal{P}^1(\xi_{1,1},h^1)+\mathcal{P}^2(\xi_{2,2},h^2)-\frac{1}{2}\left(\phi_{v_1,v_2}(u)-u\right)^2\\
&+(\theta_{1,1}(v_1)-\theta_{1,1}(c_1))_++(\theta_{2,2}(v_2)-\theta_{2,2}(c_2))_+.
\end{split}
\end{align}
\end{proposition}

\begin{Proof of theorem} {\bf \ref{sec2:thm1}:} Since the proof for the three cases of $u_f$ are the same, we will only present the details for the case $u_f=\sqrt{c_1c_2}.$ For $\varepsilon>0$, since $u_f$ is the unique fixed point of $\phi_{c_1,c_2},$ it implies 
\begin{align*}
\varepsilon_1^*:=\frac{1}{2}\min\{|\phi_{c_1,c_2}(u)-u|^2:u\in S_\varepsilon(u_f,c_1,c_2)\}>0.
\end{align*}
Recall $S_N:=\{i/N:-N\leq i\leq N\}.$ Taking $v_1=c_1,v_2=c_2$ and applying $(\ref{sec2:prop2:eq1}),$ we have that 
\begin{align}
\label{sec2:thm1:proof:eq1}
p_{N,u}\leq \mathcal{P}^1(\xi_{1,1},h^1)+\mathcal{P}^2(\xi_{2,2},h^2)-\varepsilon_1^*
\end{align}
for all $N\geq 1$ and $u\in S_N\cap S_\varepsilon(u_f,c_1,c_2).$ Let us observe the following facts:
\begin{itemize}
\item From Proposition \ref{add:prop1} $(a)$, the mapping $(u,v_1,v_2)\mapsto \phi_{v_1,v_2}(u)$ is a continuous function on the space $$\{(u,v_1,v_2):0\leq v_1,v_2\leq 1,|u|\leq\sqrt{v_1v_2}\}.$$
\item From Proposition \ref{intro:prop1}, $(\phi_{c_1,c_2}(-u_f)+u_f)^2>0.$
\item $\lim_{v_j\downarrow c_j}(\theta_{j,j}(v_j)-\theta_{j,j}(c_j))_+=0$ for $j=1,2$.
\end{itemize} 
They together imply that there exist $c_1<v_1<1$, $c_2<v_2<1$, and $\varepsilon_2^*>0$ such that
\begin{align*}
\frac{1}{2}(\phi_{v_1,v_2}(u)-u)^2-(\theta_{1,1}(v_1)-\theta_{1,1}(c_1))_+-(\theta_{2,2}(v_2)-\theta_{2,2}(c_2))_+\geq \varepsilon_2^*
\end{align*}
for all $u\in S_N$ with $-\sqrt{v_1v_2}\leq u\leq -\sqrt{c_1c_2}$. Consequently from $(\ref{sec2:prop2:eq1}),$ these $u$'s satisfy for $N\geq 1,$
\begin{align*}
p_{N,u}\leq \mathcal{P}^1(\xi_{1,1},h^1)+\mathcal{P}^2(\xi_{2,2},h^2)-\varepsilon_2^*.
\end{align*}
Let $\varepsilon^*=\min(\varepsilon_1^*,\varepsilon_2^*).$ This and $(\ref{sec2:thm1:proof:eq1})$ conclude that for $N\geq 1,$
\begin{align}
\label{sec2:thm1:proof:eq3}
p_{N,u}\leq \mathcal{P}^1(\xi_{1,1},h^1)+\mathcal{P}^2(\xi_{2,2},h^2)-\varepsilon^*
\end{align}
for $u\in S_N\cap S_\varepsilon(u_f,v_1,v_2).$ Using Parisi's formula \eqref{pf}, the free energies $p_{N}^1,p_N^2$ for the two systems satisfy 
$
\mathcal{P}^1(\xi_{1,1},h^1)+\mathcal{P}^2(\xi_{2,2},h^2)<p_N^1+p_{N}^2+{\varepsilon^*}/{4}
$
for large $N.$ It follows from \eqref{sec2:thm1:proof:eq3} that 
\begin{align}
\label{eq4}
p_{N,u}\leq p_N^1+p_N^2-\frac{3\varepsilon^*}{4},\,\,\forall u\in S_N\cap S_\varepsilon(u_f,v_1,v_2).
\end{align}

To finish our proof, we will need concentration inequalities for the free energies with respect to the two major sources of randomness, $\mathcal{G}^1$, $\mathcal{G}^2$ and $\mathbf{h}^1:=(h_i^1)_{i\leq N}$, $\mathbf{h}^2:=(h_i^2)_{i\leq N}$. Since $X_N^1$ and $X_N^2$ are jointly Gaussian, there exists a $M$-dimensional standard Gaussian r.v. $\mathbf{g}$ and a vector $\mathbf{x}(\vsi,\vtau)\in\mathbb{R}^M$ such that $(\mathbf{g}\cdot \mathbf{x}(\vsi,\vtau):\vsi,\vtau)$ has the same distribution as the family $(X_N^1(\vsi)+X_N^2(\vtau):\vsi,\vtau).$ Let $u\in S_N.$ Define
\begin{align*}
F(\mathbf{g},\mathbf{y}^1,\mathbf{y}^2)&=\frac{1}{N}\log\sum_{R(\vsi,\vtau)=u}\exp\biggl(\mathbf{g}\cdot\mathbf{x}(\vsi,\vtau)+\mathbf{y}^1\cdot\vsi+\mathbf{y}^2\cdot\vtau\biggr)
\end{align*}
for $(\mathbf{g},\mathbf{y}^1,\mathbf{y}^2)\in \mathbb{R}^M\times\mathbb{R}^N\times\mathbb{R}^N.$ We denote by $\p_\mathbf{g}$ and $\e_\mathbf{g}$ the probability and expectation with respect to only the randomness $\mathbf{g}.$ Similarly, $\p_{\mathbf{h}}$ and $\e_{\mathbf{h}}$ are defined for $\mathbf{h}^1,\mathbf{h}^2.$ Note that 
\begin{align*}
\|\mathbf{x}(\vsi,\vtau)\|^2&=\e(H_N^1(\vsi)+H_N^2(\vtau))^2\\
&=N(\xi_{1,1}(1)+\xi_{2,2}(1)+2\xi_{1,2}(R(\vsi,\vtau)))\\
&\leq NK^2,
\end{align*}
where $K:=(\xi_{1,1}(1)+\xi_{2,2}(1)+2\xi_{1,2}(1))^{1/2}.$ From this, the Cauchy-Schwartz inequality yields that 
\begin{align*}
|\mathbf{g}\cdot\mathbf{x}(\vsi,\vtau)-\mathbf{g}'\cdot\mathbf{x}(\vsi,\vtau)|&\leq\|\mathbf{x}(\vsi,\vtau)\|\|\mathbf{g}-\mathbf{g}'\|\leq \sqrt{N}K\|\mathbf{g}-\mathbf{g}'\|,\,\,\forall \mathbf{g},\mathbf{g}',
\end{align*}
and thus $F(\cdot,\mathbf{y}^1,\mathbf{y}^2)$ is Lipschitz with constant $K/\sqrt{N}$ for any $\mathbf{y}^1,\mathbf{y}^2$. Applying the Gaussian concentration inequality (see Theorem 1.4.3 \cite{Tal11}) gives
\begin{align}\label{eq1}
\p_{\mathbf{g}}\biggl(\bigl|F(\mathbf{g},\mathbf{y}^1,\mathbf{y}^2)-\e_\mathbf{g}F(\mathbf{g},\mathbf{y}^1,\mathbf{y}^2)\bigr|\geq \frac{\varepsilon^*}{16}\biggr)\leq K_1\exp \biggl(-\frac{N}{K_1}\biggr),
\end{align}
where $K_1>0$ is independent of $u,$ $N,$ $\mathbf{y}^1,\mathbf{y}^2$. Using the Cauchy-Schwartz inequality again, one also has that for $1\leq i\leq N,$
\begin{align*}
&|(\mathbf{y}^1\cdot\vsi+\mathbf{y}^2\cdot\vtau)-(\mathbf{y}_i^1(z^1)\cdot\vsi+\mathbf{y}_i^2(z^2)\cdot\vtau)|\\
&=|y_i^1\sigma_i+y_i^2\tau_i-z^1\sigma_i-z^2\tau_i|\\
&\leq |y_i^1|+|y_i^2|+|z^1|+|z^2|
\end{align*}
and so
\begin{align*}
|F(\mathbf{g},\mathbf{y}^1,\mathbf{y}^2)-F(\mathbf{g},\mathbf{y}_i^1(z^1),\mathbf{y}_i^2(z^2))|\leq \frac{1}{N}(|y_i^1|+|y_i^2|+|z^1|+|z^2|)
\end{align*}
for all $\mathbf{g}\in\mathbb{R}^M$, $\mathbf{y}^1=(y_1^1,\ldots,y_N^1)$, $\mathbf{y}^2=(y_1^2,\ldots,y_N^2)\in\mathbb{R}^N,$ and $z^1,z^2\in\mathbb{R},$ where $\mathbf{y}_i^j(z^j):=(y_1^j,\ldots,y_{i-1}^{j},z^j,y_{i+1}^{j},\ldots,y_N^j).$ If we define $F(\mathbf{y}^1,\mathbf{y}^2)=\e_{\mathbf{g}}F(\mathbf{g},\mathbf{y}^1,\mathbf{y}^2),$ then 
\begin{align*}
&|F(\mathbf{y}^1,\mathbf{y}^2)-F(\mathbf{y}_i^1(z^1),\mathbf{y}_i^2(z^2))|\leq \frac{1}{N}(|y_i^1|+|y_i^2|+|z^1|+|z^2|).
\end{align*}
This inequality allows us to use the concentration inequality for correlated
sub-Gaussian r.v. (see Proposition \ref{appendix:prop1} in Appendix),
\begin{align}\label{eq2}
\p_\mathbf{h}\biggl(\bigl|F(\mathbf{h}^1,\mathbf{h}^2)-\e_{\mathbf{h}}F(\mathbf{h}^1,\mathbf{h}^2)\bigr|\geq \frac{\varepsilon^*}{16}\biggr)\leq K_2\exp\biggl(-\frac{N}{K_2}\biggr),
\end{align}
where $K_2>0$ is independent of $u$ and $N$. Putting \eqref{eq1} and \eqref{eq2} together, using triangle inequality, and the independence between $\mathbf{g}$ and $\mathbf{h}^1,\mathbf{h}^2$ give
\begin{align}
\begin{split}\label{eq3}
&\p\biggl(\bigl|F(\mathbf{g},\mathbf{h}^1,\mathbf{h}^2)-\e F(\mathbf{g},\mathbf{h}^1,\mathbf{h}^2)\bigr|\geq \frac{\varepsilon^*}{8}\biggr)\leq K_3\exp\biggl(-\frac{N}{K_3}\biggr),
\end{split}
\end{align}
where $K_3:=K_1+K_2.$ In other words, with probability $\geq 1-K_3\exp(-N/K_3),$
$$
\biggl|\frac{1}{N}\log\sum_{R(\vsi,\vtau)=u}\exp\bigl(H_N^1(\vsi)+H_N^2(\vtau)\bigr)-p_{N,u}\biggr|<\frac{\varepsilon^*}{8}.
$$
A similar argument also yields that with probability $\geq 1-K_4\exp(-N/K_4)$, we have
\begin{align*}
\max\biggl\{\biggl|\frac{1}{N}\log Z_N^1-p_{N}^1\biggr|,\biggl|\frac{1}{N}\log Z_N^2-p_{N}^2\biggr|\biggr\}<\frac{\varepsilon^*}{8},
\end{align*}
where $K_4>0$ is independent of $N.$ From \eqref{eq4}, combining these inequalities together and noting that $S_N$ contains at most $2N+1$ numbers, we conclude that for large $N$, 
the following holds
\begin{align*}
\bigl<I\bigl(R(\vsi,\vtau)\in S_N\cap S_\varepsilon(u_f,v_1,v_2)\bigr)\bigr>\leq \exp\biggl(-\frac{3\varepsilon^*N}{8}\biggr)
\end{align*}
with probability at least $$
1-\bigl(2K_4\exp(-N/K_4)+(2N+1)K_3\exp(-N/K_3)\bigr).
$$ 
This clearly gives our assertion.
\end{Proof of theorem}

\begin{Proof of proposition} {\bf \ref{sec2:prop2}:} Consider two triplets $(k,\mathbf{m},\mathbf{q}^1)$ and $(k,\mathbf{m},\mathbf{q}^2)$. Suppose that $q_\iota^1=v_1$ and $q_{\iota}^2=v_2$ for some $\iota$ with $1\leq \iota\leq k+1.$ From Proposition $\ref{sec1:prop}$, we obtain the bound $(\ref{sec1:lem1:eq1}).$ Let $\delta_1$ and $\delta_2$ satisfy $0<\delta_1<c_1$ and $0<\delta_2<c_2.$ For convenience, we denote by $C_1$ the first term and $C_2$ the second term of the last line of the inequality $(\ref{sec1:lem1:eq1}).$ If $q_\iota^j\leq c_j-\delta_j,$ then 
\begin{align*}
C_j&\leq \max\{m_p:q_p^j\leq c_j-\delta_j\}\sum_{0\leq p\leq \iota-1}(\theta_{j,j}(q_{p+1}^j)-\theta_{j,j}(q_p^j))\\
&\leq {\mu}^j([0,c_j-\delta_j])\theta_{j,j}(1);
\end{align*}
if $q_\iota^j>c_j-\delta_j,$ then
\begin{align*}
C_j&\leq  \max\{m_p:q_p^j\leq c_j-\delta_j\}\sum_{0\leq p\leq \iota-1}(\theta_{j,j}(q_{p+1}^j)-\theta_{j,j}(q_p^j))\\
&+\sum_{0\leq p\leq \iota-1:q_p^j>c_j-\delta_j}(\theta_{j,j}(q_{p+1}^j)-\theta_{j,j}(q_p^j))\\
&\leq {\mu}^j([0,c_j-\delta_j])\theta_{j,j}(1)+\theta_{j,j}(v_j)-\theta_{j,j}(c_j-\delta_j).
\end{align*}
As a summary, we have
\begin{align}
\label{sec2:prop2:proof:eq2}
C_j\leq {\mu}^j([0,c_j-\delta_j])\theta_{j,j}(1)+(\theta_{j,j}(v_j)-\theta_{j,j}(c_j-\delta_j))_+.
\end{align} 
Finally, combining $(\ref{sec1:lem1:eq1})$ and $(\ref{sec2:prop2:proof:eq2})$, we obtain that 
\begin{align}
\begin{split}\label{add:eq15}
p_{N,u}&\leq \mathcal{P}_{k}^1(\mathbf{m},\mathbf{q}^1)+\mathcal{P}_{k}^2(\mathbf{m},\mathbf{q}^2)\\
&-\frac{1}{2}\left(\e\frac{\partial\Phi_{1,\mu^1}}{\partial x}(h^1+\chi^1,v_1)\frac{\partial\Phi_{2,\mu^2}}{\partial x}(h^2+\chi^2,v_2)-u\right)^2\\
&+\mu^1([0,c_1-\delta_1])\theta_{1,1}(1)+(\theta_{1,1}(v_1)-\theta_{1,1}(c_1-\delta_1))_+\\
&+\mu^2([0,c_2-\delta_2])\theta_{2,2}(1)+(\theta_{2,2}(v_2)-\theta_{2,2}(c_2-\delta_2))_+.
\end{split}
\end{align}
where $\chi_1$ and $\chi_2$ are jointly centered Gaussian with $\e\chi_1^2=\xi_{1,1}'(v_1),$ $\e\chi_2^2=\xi_{2,2}'(v_2),$ and $\e \chi_1\chi_2=\xi_{1,2}'(u).$ 

\smallskip

Finally, take two sequences of triplets $(k_{n},\mathbf{m}_n,\mathbf{q}_n^1)_{n\geq 1}$ and $(k_{n},\mathbf{m}_n,\mathbf{q}_n^2)_{n\geq 1}$ such that their corresponding probability measures $(\mu_n^1)_{n\geq 1}$ and $(\mu_n^2)_{n\geq 1}$ converge weakly to $\mu_P^1$ and $\mu_P^2,$ respectively. We may also require that $q_{\iota_n}^1=v_1$ and $q_{\iota_n}^2=v_2$ for some $1\leq \iota_n\leq k_n+1$ for each $n.$ Applying these triplets to $(\ref{add:eq15})$, $(a)$ in Proposition \ref{add:prop1}, and then letting $\delta_j\downarrow 0$, the asserted result follows.
\end{Proof of proposition}



\section{Controlling overlaps using identities}\label{identity}

Recall from Section \ref{sec:intro} that $G_N^1$ and $G_N^2$ are the Gibbs measures corresponding to the Hamiltonians $H_N^1$ and $H_N^2$ as in $(\ref{intro:eq18})$ using temperature $\vec{\beta}_1$ and $\vec{\beta}_2$, disorder $\mathcal{G}^1$ and $\mathcal{G}^2,$ and external field $h^1$ and $h^2,$ respectively. Throughout this section, we will assume that $h^1$ and $h^2$ are jointly Gaussian (might not be centered). Recall that $c_j$ is the minimum value of the support of the Parisi measure $\mu_{P}^j$ for $j=1,2.$ The major goal of this section is to prove that under the assumptions of Theorems \ref{thm:tc} and $\ref{thm:dc},$ the following theorems hold that will be used in the problems of chaos in temperature and disorder. Recall the definitions of $(t_p)_{p\geq 1},$ $\mathcal{C}_0,$ $\mathcal{I}_1,$ $\mathcal{I}_2,$ $(C_1),$ and $(C_2)$ from Section \ref{sec:intro}. 

\begin{theorem}
\label{GG:thm01}
Let $t_{p}=1$ for all $p\in\mathbb{N}$. Suppose that $\mathcal{I}_1$ and $\mathcal{I}_2$ satisfy $(C_1)$ and $(C_2),$ respectively. Then for $j=1,2,$
\begin{align}
\label{GG:thm01:eq1}
\lim_{N\rightarrow\infty}\e\left<I(|R(\vsi,\vtau)|>\sqrt{c_j}+\varepsilon)\right>=0,\,\,\forall \varepsilon>0.
\end{align}
If $\mbox{\rm Var}(h^j)\neq 0$ for both $j=1,2$, then 
\begin{align}
\label{GG:thm01:eq2}
\lim_{N\rightarrow\infty}\e\left<I(|R(\vsi,\vtau)|>\sqrt{c_1c_2}+\varepsilon)\right>=0,\,\,\forall \varepsilon>0.
\end{align}
\end{theorem}

\begin{theorem}
\label{GG:thm02}
Suppose that $0\leq t_{p}<1$ for some $p\in \mathcal{I}_1\cap\mathcal{I}_2.$ For $j\in\{1,2\},$ if $\mathcal{I}_j\in\mathcal{C}_0$, then $(\ref{GG:thm01:eq1})$ holds. If $\mbox{\rm Var}(h^j)\neq 0$ and $\mathcal{I}_j\in\mathcal{C}_0$ for both $j=1,2$, then we have $(\ref{GG:thm01:eq2})$.
\end{theorem}

The importance of Theorems \ref{GG:thm01} and \ref{GG:thm02} lies on the fact that they allow us to exclude the discussion on the cases $|u|>\sqrt{c_j}$ when $\e(h^j)^2=0$ and $|u|>\sqrt{c_1c_2}$ when $\mbox{Var}(h^1)^2\neq 0$, $\mbox{Var}(h^2)^2\neq 0$ in the control of the coupled free energy $p_{N,u}$ using Guerra's bound, which are technically very hard to deal with. Our approach to Theorems \ref{GG:thm01} and \ref{GG:thm02} is intimately motivated by \cite{CP12}. As we have explained in Section $1$, we will derive the Ghirlanda-Guerra identities as well as a new family of identities for the overlaps in the coupled system to control the cross overlap $R(\vsi,\vtau)$ using the overlaps $R(\vsi^1,\vsi^2)$ and $R(\vtau^1,\vtau^2)$ from each individual system. Let us remark that in the statements of Theorems $\ref{GG:thm01}$ and $\ref{GG:thm02}$, the first parts $(\ref{GG:thm01:eq1})$ have been considered in Theorems 3, 4 \cite{CP12}, while the second parts $(\ref{GG:thm01:eq2})$ are new that strongly rely on the positivity of the overlap $(\ref{intro:eq19}).$

\subsection{Identities for the coupled system}
Given replicas $(\vsi^\ell,\vtau^\ell)_{\ell\geq 1},$ let us denote by
\begin{align*}
R_{\ell,\ell'}^1=R(\vsi^\ell,\vsi^{\ell'}),\,\,R_{\ell,\ell'}^2=R(\vtau^\ell,\vtau^{\ell'}),\,\,R_{\ell,\ell'}=R(\vsi^\ell,\vtau^{\ell'}),
\end{align*}
the overlaps within each system and between the two systems. For any bounded function $f$ depending only on the overlaps $(R_{\ell,\ell'}^1)_{\ell,\ell'\leq n}$, $(R_{\ell,\ell'}^2)_{\ell,\ell'\leq n},$ and $(R_{\ell,\ell'})_{\ell,\ell'\leq n}$ and any $\psi\in C[-1,1],$ we define
\begin{align}
\begin{split}
\label{GG:eq1}
\Phi_{1,n}(f,\psi)&=\e\left<f\psi(R_{1,n+1}^1)\right>-\frac{1}{n}\e\left<f\right>\e\left<\psi(R_{1,2}^1)\right>-\frac{1}{n}\sum_{\ell=2}^n\e\left<f\psi(R_{1,\ell}^1)\right>,
\end{split}\\
\begin{split}
\label{GG:eq2}
\Psi_{1,n}(f,\psi)&=\e\left<f\psi(R_{1,n+1})\right>-\frac{1}{n}\sum_{\ell=1}^n \e\left<f\psi(R_{1,\ell})\right>,
\end{split}\\
\begin{split}
\label{GG:eq3}
\Phi_{2,n}(f,\psi)&=\e\left<f\psi(R_{1,n+1}^2)\right>-\frac{1}{n}\e\left<f\right>\e\left<\psi(R_{1,2}^2)\right>-\frac{1}{n}\sum_{\ell=2}^n\e\left<f\psi(R_{1,\ell}^2)\right>,
\end{split}\\
\begin{split}
\label{GG:eq4}
\Psi_{2,n}(f,\psi)&=\e\left<f\psi(R_{n+1,1})\right>-\frac{1}{n}\sum_{\ell=1}^n \e\left<f\psi(R_{\ell,1})\right>.
\end{split}
\end{align} 
In what follows, we will prove that these four quantities converge to zero as $N$ tends to infinity for either all even $\psi\in C[-1,1]$ or all $\psi\in C[-1,1]$ depending on the parameters of the models. Equations $(\ref{GG:eq1})$ and $(\ref{GG:eq3})$ will yield the familiar Ghirlanda-Guerra identities \cite{GG98}, only now the function $f$ may depend on the overlaps of the two systems. As for equations $(\ref{GG:eq2})$ and $(\ref{GG:eq4})$, they will provide additional information about how two systems interact with each other. We will use the notation throughout the section:
$$
\psi_{a}(x)=x^{a}.
$$
Write $(h^1,h^2)=(x^1+s_1g^1,x^2+s_2g^2)$, where $x^1,x^2\in\mathbb{R}$, $s_1=\mbox{Var}(h^1)^{1/2},s_2=\mbox{Var}(h^2)^{1/2}$, and $g^1,g^2$ are jointly centered Gaussian with $\e(g^1)^2=\e(g^2)^2=1$ and $\e g^1g^2=t$ for some $t\in [-1,1].$ 
The following lemma is the key to establish the asserted identities for the coupled system.

\begin{lemma}\label{GG:lem1} For all $p\in \mathbb{N},$ 
\begin{align}
\begin{split}\label{GG:lem1:eq1}
&\lim_{N\rightarrow\infty}\sup_{\|f\|_\infty\leq 1}|\beta_{2,p}\sqrt{1-t_{p}}\Psi_{1,n}(f,\psi_{2p})|=0,
\end{split}\\
\begin{split}\label{GG:lem1:eq2}
&\lim_{N\rightarrow\infty}\sup_{\|f\|_\infty\leq 1}|\beta_{1,p}\sqrt{1-t_{p}}\Psi_{2,n}(f,\psi_{2p})|=0,
\end{split}\\
\begin{split}\label{GG:lem1:eq3}
&\lim_{N\rightarrow\infty}\sup_{\|f\|_\infty\leq 1}|\beta_{1,p}\Phi_{1,n}(f,\psi_{2p})+\beta_{2,p}t_{p}\Psi_{1,n}(f,\psi_{2p})|=0,
\end{split}\\
\begin{split}\label{GG:lem1:eq4}
&\lim_{N\rightarrow\infty}\sup_{\|f\|_\infty\leq 1}|\beta_{2,p}\Phi_{2,n}(f,\psi_{2p})+\beta_{1,p}t_{p}\Psi_{2,n}(f,\psi_{2p})|=0.
\end{split}
\end{align}
We also have
\begin{align}
\begin{split}\label{GG:lem1:eq5}
&\lim_{N\rightarrow\infty}\sup_{\|f\|_\infty\leq 1}|s_2\sqrt{1-|t|}\Psi_{1,n}(f,\psi_1)|=0,
\end{split}\\
\begin{split}\label{GG:lem1:eq6}
&\lim_{N\rightarrow\infty}\sup_{\|f\|_\infty\leq 1}|s_1\sqrt{1-|t|}\Psi_{2,n}(f,\psi_{1})|=0,
\end{split}\\
\begin{split}\label{GG:lem1:eq7}
&\lim_{N\rightarrow\infty}\sup_{\|f\|_\infty\leq 1}|s_1\Phi_{1,n}(f,\psi_{1})+s_2t\Psi_{1,n}(f,\psi_{1})|=0,
\end{split}\\
\begin{split}\label{GG:lem1:eq8}
&\lim_{N\rightarrow\infty}\sup_{\|f\|_\infty\leq 1}|s_2\Phi_{2,n}(f,\psi_{1})+s_1t\Psi_{2,n}(f,\psi_{1})|=0.
\end{split}
\end{align}
\end{lemma}

\begin{proof} Our proof basically follows the same argument as Lemma 2 \cite{CP12}. Let $X_{N,p}^1(\vsi)$ and $X_{N,p}^2(\vtau)$ be the pure $2p$-spin Hamiltonian in $X_N^1(\vsi)$ and $X_N^2(\vtau).$ They are equal in distribution to the pair
$$
\sqrt{t_{p}}X_{N,p}(\vsi)+\sqrt{1-t_{p}}Z_{N,p}^1(\vsi)\,\,\mbox{and}\,\,\sqrt{t_{p}}X_{N,p}(\vtau)+\sqrt{1-t_{p}}Z_{N,p}^2(\vtau),
$$ 
where we denote by $X_{N,p}$, $Z_{N,p}^1$, $Z_{N,p}^2$ three independent copies of $(\ref{intro:eq17}).$ The derivation of $(\ref{GG:lem1:eq1})-(\ref{GG:lem1:eq4})$ is based on the concentration of the Hamiltonians (see Lemma 1 \cite{CP12} and also Chapter 12 \cite{Tal11}): as $N\rightarrow \infty$,
\begin{align}\label{GG:lem1:proof:eq1}
\Delta_{p}^1,\Delta_{p}^2,\Gamma_{p}^1,\Gamma_{p}^2\rightarrow 0,
\end{align} 
where
\begin{align*}
\Delta_{p}^1&=N^{-1}\e\left<\left|Z_{N,p}^2(\vsi^1)-\e\left<Z_{N,p}^2(\vsi^1)\right>\right|\right>,\\
\Delta_{p}^2&=N^{-1}\e\left<\left|Z_{N,p}^1(\vtau^1)-\e\left<Z_{N,p}^1(\vtau^1)\right>\right|\right>,\\
\Gamma_{p}^1&=N^{-1}\e\left<\left|X_{N,p}^1(\vsi^1)-\e\left<X_{N,p}^1(\vsi^1)\right>\right|\right>,\\
\Gamma_{p}^2&=N^{-1}\e\left<\left|X_{N,p}^2(\vtau^1)-\e\left<X_{N,p}^2(\vtau^1)\right>\right|\right>.
\end{align*}
For $f$ with $\|f\|_\infty\leq 1$, one may see clearly
\begin{align}
\begin{split}
\label{GG:lem1:proof:eq2}
&N^{-1}\left|\e\left<Z_{N,p}^2(\vsi^1)f\right>-\e\left<Z_{N,p}^2(\vsi^1)\right>\e\left<f\right>\right|\leq \Delta_{p}^1,\\
&N^{-1}\left|\e\left<Z_{N,p}^1(\vtau^1)f\right>-\e\left<Z_{N,p}^1(\vtau^1)\right>\e\left<f\right>\right|\leq \Delta_{p}^2,\\
&N^{-1}\left|\e\left<X_{N,p}^1(\vsi^1)f\right>-\e\left<X_{N,p}^1(\vsi^1)\right>\e\left<f\right>\right|\leq \Gamma_{p}^1,\\
&N^{-1}\left|\e\left<X_{N,p}^2(\vtau^1)f\right>-\e\left<X_{N,p}^2(\vtau^1)\right>\e\left<f\right>\right|\leq \Gamma_{p}^2.
\end{split}
\end{align}
A simple application of the Gaussian integration by parts to each term of the left-hand side of $(\ref{GG:lem1:proof:eq2})$ together with $(\ref{GG:lem1:proof:eq1})$ yields $(\ref{GG:lem1:eq1})$, $(\ref{GG:lem1:eq2})$, $(\ref{GG:lem1:eq3})$, and $(\ref{GG:lem1:eq4}).$ One may refer to Lemma 2 \cite{CP12} for detail. Similarly, since $h^1,h^2$ are jointly Gaussian, $(\ref{GG:lem1:eq5})$, $(\ref{GG:lem1:eq6})$, $(\ref{GG:lem1:eq7})$, and $(\ref{GG:lem1:eq8})$ can also be treated by applying the same argument as above to the external fields.
\end{proof}

We will need the following lemma.

\begin{lemma}
\label{GG:lem2}
Let $j\in\{1,2\}.$ Suppose that
\begin{align}
\label{GG:lem2:eq1}
\lim_{N\rightarrow\infty}\sup_{\|f\|_\infty\leq 1}|\Psi_{j,n}(f,\psi)|=0
\end{align}
holds with $\psi=\psi_{a}$ for some $a\geq 1.$ If $a\in 2\mathbb{N}$, then $(\ref{GG:lem2:eq1})$ also holds for all even $\psi\in C[-1,1];$ if $a=1,$ then $(\ref{GG:lem2:eq1})$ holds for all $\psi\in C[-1,1].$ 
\end{lemma}

\begin{proof}
It suffices to consider $j=1.$ Observe that for $\ell\geq 2,$ using symmetry between replicas yields
\begin{align}\label{GG:lem2:proof:eq1}
\e\left<((R_{1,1})^a-(R_{1,\ell})^a)^2\right>&=2\e\left<(R_{1,1})^{2a}\right>-2\e\left<(R_{1,1})^{a}(R_{1,2})^{a}\right>=-2\Psi_{1,1}(f,\psi_a)
\end{align} 
by definition of $\Psi_{1,n}$ in $(\ref{GG:eq2})$ with $n=1$ and $f=(R_{1,1})^a.$ If $a\in 2\mathbb{N},$ using $|x-y|^a\leq |x^a-y^a|$ for all $x,y\geq 0$ and $(\ref{GG:lem2:proof:eq1})$, we can write
\begin{align*}
\e\left<||R_{1,1}|-|R_{1,\ell}||\right>&\leq \left(\e\left<||R_{1,1}|-|R_{1,\ell}||^{2a}\right>\right)^{1/2a}\\
&\leq \left(\e\left<((R_{1,1})^a-(R_{1,\ell})^a)^2\right>\right)^{1/2a}=(-2\Psi_{1,1}(f,\psi_a))^{1/2a}.
\end{align*}
Since $(\ref{GG:lem2:eq1})$ holds for $\psi_a,$ this inequality implies that $|R_{1,\ell}|\approx |R_{1,1}|$ for all $\ell\geq 2$ and clearly $(\ref{GG:lem2:eq1})$ holds for all even $\psi \in C[-1,1].$ If $(\ref{GG:lem2:eq1})$ holds for $\psi_1,$ then $(\ref{GG:lem2:proof:eq1})$ implies $R_{1,1}\approx R_{1,\ell}$ for all $\ell\geq 1$ and so $(\ref{GG:lem2:eq1})$ holds for all $\psi\in C[-1,1].$ This completes our proof.
\end{proof}

Recall the positivity of the overlap $(\ref{intro:eq19})$ that if $\e(h^j)^2\neq 0,$ one may pass to limit to see
\begin{align}
\begin{split}\label{GG:prop3:proof:eq4}
\lim_{N\rightarrow\infty}\e\left<I(R_{1,\ell}^j\geq 0,\,\forall 1\leq \ell\leq n)\right>=1,\,\forall n\geq 1.
\end{split}
\end{align}
We continue to state two useful propositions that will need the help of Lemmas $\ref{GG:lem1}$, $\ref{GG:lem2}$, and $(\ref{GG:prop3:proof:eq4})$ under additional assumptions on the parameters of the models. 
 
\begin{proposition}
\label{GG:prop3}
Suppose that $t_p=1$ for all $p\geq 1.$ For $j\in\{1,2\},$ if $\mathcal{I}_j$ satisfies $(C_j),$ then $(\ref{GG:lem2:eq1})$ and
\begin{align}
\label{GG:prop3:eq1}
\lim_{N\rightarrow\infty}\sup_{\|f\|_\infty\leq 1}|\Phi_{j,n}(f,\psi)|=0
\end{align}
hold for all even $\psi\in C[-1,1]$. If $\mathcal{I}_j$ satisfies $(C_j)$ and $\mbox{\rm Var}(h^j)\neq 0$ for both $j=1,2,$ then $(\ref{GG:lem2:eq1})$ holds for both $j=1,2$ and all $\psi\in C[-1,1].$
\end{proposition}

\begin{proof}
To prove the first assertion, it suffices to consider $j=1.$ Since $\mathcal{I}_1$ satisfies condition $(C_1),$ there exist $\mathcal{A}\subseteq\mathcal{I}_1$ with $\mathcal{A}\in\mathcal{C}_0$, $p_0\in\mathcal{I}_1\setminus \mathcal{A},$ and $\nu\in\mathbb{R}$ such that $\beta_{2,p}=\nu\beta_{1,p}$ for all $p\in\mathcal{A}$ and $\beta_{2,p_0}\neq \nu\beta_{1,p_0}.$ Since $\beta_{1,p_0}\neq 0,$ $\nu':=\beta_{2,p_0}/\beta_{1,p_0}\neq \nu.$ From $(\ref{GG:lem1:eq3}),$ we have that
\begin{align}\label{GG:prop3:proof:eq1}
\lim_{N\rightarrow\infty}\sup_{\|f\|_\infty\leq 1}|\Phi_{1,n}(f,\psi_{2p_0})+\nu'\Psi_{1,n}(f,\psi_{2p_0})|=0
\end{align}
and that using $\beta_{2,p}=\nu\beta_{1,p}$ and $\beta_{1,p}\neq 0$ for all $p\in\mathcal{A}$,
\begin{align}\label{GG:prop3:proof:eq3}
\lim_{N\rightarrow\infty}\sup_{\|f\|_\infty\leq 1}|\Phi_{1,n}(f,\psi_{2p})+\nu\Psi_{1,n}(f,\psi_{2p})|=0.
\end{align}
Since $\mathcal{A}\in\mathcal{C}_0$ and $\psi_{2p_0}$ is even, we can approximate $\psi_{2p_0}$ uniformly by the linear combination of $1$ and $\psi_{2p}$ for $p\in\mathcal{A}$ to obtain
\begin{align}
\label{GG:prop3:proof:eq2}
\lim_{N\rightarrow\infty}\sup_{\|f\|_\infty\leq 1}|\Phi_{1,n}(f,\psi_{2p_0})+\nu\Psi_{1,n}(f,\psi_{2p_0})|=0.
\end{align}
From $\nu\neq \nu',$ $(\ref{GG:prop3:proof:eq1})$ and $(\ref{GG:prop3:proof:eq2})$ imply that $(\ref{GG:lem2:eq1})$ holds for $\psi_{2p_0}$ and from Lemma $\ref{GG:lem2}$, $(\ref{GG:lem2:eq1})$ holds for all even $C[-1,1].$ This together with $(\ref{GG:prop3:proof:eq3})$ implies $(\ref{GG:prop3:eq1})$ for all $\psi_{2p}$ with $p\in\mathcal{A}$ and then $\mathcal{A}\in\mathcal{C}_0$ yields $(\ref{GG:prop3:eq1})$ for all even $\psi\in C[-1,1].$ This completes the proof of the first assertion.

\smallskip

Next, suppose that $\mathcal{I}_j$ satisfies $(C_j)$ and $\mbox{Var}(h^j)\neq 0$ for both $j=1,2$. The use of $\psi_1(x)=|x|+2\min(0,x)$, the positivity of the overlaps $(\ref{GG:prop3:proof:eq4})$, and the first assertion $(\ref{GG:prop3:eq1})$ leads to
\begin{align}
\begin{split}
\label{add:eq1}
\limsup_{N\rightarrow\infty}\sup_{\|f\|_\infty\leq 1}|\Phi_{j,n}(f,\psi_1)|&\leq \limsup_{N\rightarrow\infty}\sup_{\|f\|_\infty\leq 1}|\Phi_{j,n}(f,|x|)|\\
&+2\limsup_{N\rightarrow\infty}\sup_{\|f\|_\infty\leq 1}|\Phi_{j,n}(f,\min(0,x))|=0.
\end{split}
\end{align}
Note that $s_1,s_2\neq 0$. Let us use $(\ref{GG:lem1:eq5})$ and $(\ref{GG:lem1:eq6})$ if $|t|<1$ or use $(\ref{GG:lem1:eq7}),$ $(\ref{GG:lem1:eq8}),$ and $(\ref{add:eq1})$ if $|t|=1$ to get $(\ref{GG:lem2:eq1})$ for $\psi=\psi_1$. Consequently, $(\ref{GG:lem2:eq1})$ holds for both $j=1,2$ and all $\psi\in C[-1,1]$ by Lemma \ref{GG:lem2}. This finishes our proof.
\end{proof}

\begin{proposition}\label{GG:prop1}
Suppose that $0\leq t_{p}<1$ for some $p\in\mathcal{I}_1\cap \mathcal{I}_2.$ Then $(\ref{GG:lem2:eq1})$ holds for both $j=1,2$ and all even $\psi\in C[-1,1].$ For $j\in\{1,2\},$ if $\mathcal{I}_j\in\mathcal{C}_0,$ then $(\ref{GG:prop3:eq1})$ holds for all even $\psi\in C[-1,1].$ If $\mbox{\rm Var} (h^j)\neq 0$ and $\mathcal{I}_{j}\in\mathcal{C}_0$ for both $j=1,2,$ then $(\ref{GG:lem2:eq1})$ holds for both $j=1,2$ and all $\psi\in C[-1,1].$
\end{proposition}

\begin{proof}
Since $\beta_{1,p},\beta_{2,p}\neq 0$ and $t_{p}<1,$ one may see clearly that $(\ref{GG:lem1:eq1})$, $(\ref{GG:lem1:eq2})$, and Lemma $\ref{GG:lem2}$ together imply the first assertion. Next, using the first assertion together with $(\ref{GG:lem1:eq3})$ and $(\ref{GG:lem1:eq4})$ yields that if $\mathcal{I}_j\in\mathcal{C}_0$ for some $j\in\{1,2\},$ then $(\ref{GG:prop3:eq1})$ holds for all even $\psi\in C[-1,1]$. This proves the second assertion. 

Finally, suppose that $\mbox{Var}(h^j)\neq 0$ and $\mathcal{I}_{j}\in\mathcal{C}_0$ for both $j=1,2.$ Note that $s_1,s_2\neq 0$. If $|t|<1,$ we use $(\ref{GG:lem1:eq5})$ and $(\ref{GG:lem1:eq6})$ to see that for $j=1,2$, $(\ref{GG:lem2:eq1})$ is valid for both $j=1,2$ and $\psi=\psi_1$ and from Lemma \ref{GG:lem2}, this is also true for all $\psi\in C[-1,1].$ Suppose that $|t|=1.$ Using the relation $x=|x|+2\min(x,0)$, the positivity of the overlaps $(\ref{GG:prop3:proof:eq4}),$ and the second assertion, we also get $(\ref{add:eq1})$ for both $j=1,2.$ Applying this to $(\ref{GG:lem1:eq7})$ and $(\ref{GG:lem1:eq8})$ yields that $(\ref{GG:lem2:eq1})$ for both $j=1,2$ and $\psi=\psi_1$ and thus, from Lemma $\ref{GG:lem2},$ this is also true for all $\psi\in C[-1,1].$ This completes our proof.
\end{proof}

\subsection{Proofs of Theorems \ref{GG:thm01} and \ref{GG:thm02}}

The proofs of Theorems $\ref{GG:thm01}$ and $\ref{GG:thm02}$ rely on the following two propositions.

\begin{proposition}\label{GG:prop4}
Suppose that $(\ref{GG:lem2:eq1})$ holds for both $j=1,2$ and all even $\psi\in C[-1,1].$ For $j\in\{1,2\},$ if $(\ref{GG:prop3:eq1})$ holds for all even $\psi\in C[-1,1]$, then $(\ref{GG:thm01:eq1})$ holds.
\end{proposition}

\begin{proposition}\label{GG:prop5}
For both $j=1,2$, if $(\ref{GG:lem2:eq1})$ holds for all $\psi\in C[-1,1]$ and $(\ref{GG:prop3:eq1})$ holds for all even $\psi\in C[-1,1]$, then
$(\ref{GG:thm01:eq2})$ holds.
\end{proposition}

Since the proofs of these two propositions are exactly the main ingredients of Theorems 3 and 4 in \cite{CP12}, we will only sketch the proof for Proposition $\ref{GG:prop5}$ as follows. Suppose that for both $j=1,2,$ $(\ref{GG:lem2:eq1})$ holds for all $\psi\in C[-1,1]$ and $(\ref{GG:prop3:eq1})$ holds for all even $\psi\in C[-1,1]$. Observe that from $(\ref{GG:lem2:eq1})$ and using symmetry between replicas, they essentially imply
$R_{1,1}\approx R_{\ell,\ell'}$ for all $\ell,\ell'\leq n$ and $n\geq 1.$ 
If $(\ref{GG:thm01:eq2})$ is not true, then $R_{1,1}\geq \sqrt{c_1'c_2'}$ has nonzero probability for some $c_1',c_2'$ satisfying $c_1<c_1'<1$ and $c_2<c_2'<1$. From the Ghirlanda-Guerra identities (see Lemma 4 \cite{CP12}), the following holds with nonzero probability:
$$
|R_{\ell,\ell'}|\geq \sqrt{c_1'c_2'},\,|R_{\ell,\ell'}^1|\leq c_1'',\,|R_{\ell,\ell'}^2|\leq c_2''
$$
for all $c_1'',c_2''$ satisfying $c_1<c_1''< c_1'$ and $c_2<c_2''<c_2'.$ However, using the Cauchy-Schwartz inequality to the usual inner product of $N^{-1/2}(\vsi^1+\cdots+\vsi^n)$ and $N^{-1/2}(\vtau^1+\cdots+\vtau^n)$ leads to 
$$
\sqrt{c_1'c_2'}\leq |R_{1,1}|\approx\frac{1}{n^2}\left|\sum_{\ell,\ell'=1}^nR_{\ell,\ell'}\right|\leq\left(\frac{1}{n^2}\sum_{\ell,\ell'=1}^n|R_{\ell,\ell'}^1|\right)^{1/2}\left(\frac{1}{n^2}\sum_{\ell,\ell'=1}^n|R_{\ell,\ell'}^2|\right)^{1/2}\leq \sqrt{c_1''c_2''}.
$$
This forms a contradiction since indeed $\sqrt{c_1''c_2''}<\sqrt{c_1'c_2'}.$

\smallskip
\smallskip

\begin{Proof of theorem} {\bf \ref{GG:thm01}:} For both $j=1,2$, since $\mathcal{I}_j$ satisfies $(C_j)$, it follows by Proposition \ref{GG:prop3} that $(\ref{GG:lem2:eq1})$ and $(\ref{GG:prop3:eq1})$ hold for all even $\psi\in C[-1,1]$. Thus, $(\ref{GG:thm01:eq1})$ follows for both $j=1,2$ from Proposition \ref{GG:prop4}. If, in addition, $\mbox{Var}(h^j)\neq 0$ for both $j=1,2,$ then again from Proposition \ref{GG:prop3}, $(\ref{GG:lem2:eq1})$ holds for all $\psi\in C[-1,1]$ and Proposition \ref{GG:prop5} implies (\ref{GG:thm01:eq1}).
\end{Proof of theorem}

\begin{Proof of theorem} {\bf \ref{GG:thm02}:} 
From the given condition $0\leq t_{p}<1$ for some $p\in\mathcal{I}_1\cap \mathcal{I}_2,$ we know that $(\ref{GG:lem2:eq1})$ holds for both $j=1,2$ and all even $\psi\in C[-1,1]$ by Proposition \ref{GG:prop1}. For $j\in\{1,2\},$ if $\mathcal{I}_j\in\mathcal{C}_0$, $(\ref{GG:prop3:eq1})$ is true for all even $\psi\in C[-1,1]$ by Proposition \ref{GG:prop1} and consequently, $(\ref{GG:thm01:eq1})$ holds by Proposition \ref{GG:prop4}. If $\mbox{Var}(h^j)\neq 0$ and $\mathcal{I}_j\in\mathcal{C}_0$ for both $j=1,2,$ then $(\ref{GG:lem2:eq1})$ is valid for all $\psi\in C[-1,1]$ and $(\ref{GG:prop3:eq1})$ holds for all even $\psi\in C[-1,1]$ by Proposition \ref{GG:prop1} and consequently, $(\ref{GG:thm01:eq2})$ holds by Proposition \ref{GG:prop5}.
\end{Proof of theorem}


\section{Proofs of Theorems \ref{thm:tc}, \ref{thm:dc}, and \ref{thm:efc}}

Our last section will be the proofs of Theorems $\ref{thm:tc}$, $\ref{thm:dc}$, and $\ref{thm:efc}$ that are based on our main results derived in all previous sections.

\smallskip
\smallskip

\begin{Proof of theorem} {\bf \ref{thm:tc}:}
Using the given conditions, Theorem $\ref{GG:thm01}$ implies $(\ref{GG:thm01:eq1})$ for both $j=1,2.$ For $j=1$ or $2$, if $\e(h^j)^2=0$, then from Theorem $\ref{sec3:thm1}$, $c_j=0$ and so from $(\ref{GG:thm01:eq1})$, $(\ref{thm:tc:eq1})$ follows. If $\mbox{Var}(h^j)\neq 0$ for both $j=1,2,$ then from Theorem \ref{sec2:thm1}, for $\varepsilon>0,$ there exist $c_1<v_1<1$, $c_2<v_2<1$, and $K>0$ that are independent of $N$ such that 
$(\ref{sec2:thm1:eq1})$ holds for all $N\geq 1.$ On the other hand, from Theorem \ref{GG:thm01}, we also have
\begin{align}
\label{proof:eq1}
\lim_{N\rightarrow\infty}\e\left<I(|R(\vsi,\vtau)|>\sqrt{v_1'v_2'})\right>=0
\end{align}
for all $c_1<v_1'<v_1$ and $c_2<v_2'<v_2.$ Combining this with $(\ref{sec2:thm1:eq1})$ gives $(\ref{thm:tc:eq2}).$
\end{Proof of theorem}

\begin{Proof of theorem} {\bf \ref{thm:dc}:}
This part of the proof is very similar to that for Theorem \ref{thm:tc}. Suppose that $\e(h^j)^2=0$ and $\mathcal{I}_j\in\mathcal{C}_0$ for some $j=1,2.$ From Theorems \ref{sec3:thm1} and \ref{GG:thm02}, $(\ref{thm:tc:eq1})$ follows. Suppose that $\mbox{Var}(h^j)\neq 0$ and $\mathcal{I}_j\in\mathcal{C}_j$ for both $j=1,2.$ From Theorem $\ref{sec2:thm1}$, there are $c_1<v_1<1$, $c_2<v_2<1$, and $K>0$ independent of $N$ such that $(\ref{sec2:thm1:eq1})$ holds for all $N\geq 1.$ Also, from Theorem \ref{GG:thm02}, for any $(v_1',v_2')$ with $c_1<v_1'<v_1$ and $c_2<v_2'<v_2$, we have $(\ref{proof:eq1}).$ This together with $(\ref{sec2:thm1:eq1})$ implies $(\ref{thm:tc:eq2}).$
\end{Proof of theorem}

\begin{Proof of theorem} {\bf \ref{thm:efc}:} Since $\beta_{1,p}=\beta_{2,p}$, $t_{p}=1$ for all $p\geq 1$, and $h^1,h^2$ are identically distributed, the two systems are equal to each other in distribution. Thus, we may pick $\mu_{P}^1=\mu_{P}^2$ and simply denote them by $\mu_P.$ Let $\xi:=\xi_{1,1}=\xi_{2,2}=\xi_{1,2}$ and $c:=\min\mbox{supp}\mu_P$. Note that $c>0$ since $\e(h^1)^2=\e(h^2)\neq 0$. Let $u_f$ be the fixed point of $\phi_{c,c}$ from Proposition $\ref{intro:prop1}.$ We claim that $|u_f|<c.$ If $u_f=c,$ then using $\phi_{c,c}(u_f)=u_f$ and $(\ref{sec2:prop1:eq1})$ implies
\begin{align*}
&\e\left(\frac{\partial\Phi_{\mu_P}}{\partial x}(h^1+\chi,c)-\frac{\partial\Phi_{\mu_P}}{\partial x}(h^2+\chi,c)\right)^2=2c-2u_f=0,
\end{align*}
where $\chi$ is centered Gaussian with variance $\xi'(c)$ independent of $h^1,h^2.$ 
This means that $$
\frac{\partial\Phi_{\mu_P}}{\partial x}(h^1+\chi,c)=\frac{\partial\Phi_{\mu_P}}{\partial x}(h^2+\chi,c)\,\,a.s.
$$
However, since $\frac{\partial\Phi_{\mu_P}}{\partial x}(\cdot,c)$ is strictly increasing from $(b)$ in Proposition \ref{add:prop1}, we obtain $h^1+\chi=h^2+\chi$ a.s. and thus, $h^1=h^2$ a.s. forms a contradiction. Similarly, if $u_f=-c,$ then using $\phi_{c,c}(u_f)=u_f$ and $(\ref{sec2:prop1:eq1})$ yields
\begin{align*}
&\e\left(\frac{\partial\Phi_{\mu_P}}{\partial x}(h^1+\chi,c)+\frac{\partial\Phi_{\mu_P}}{\partial x}(h^2-\chi,c)\right)^2=2c+2u_f=0,
\end{align*}
where $\chi$ is defined as above. This means that $$
\frac{\partial\Phi_{\mu_P}}{\partial x}(h^1+\chi,c)=-\frac{\partial\Phi_{\mu_P}}{\partial x}(h^2-\chi,c)\,\,a.s.
$$
Since $\frac{\partial\Phi_{\mu_P}}{\partial x}(\cdot,c)$ is odd and strictly increasing from $(b)$ and $(c)$ in Proposition \ref{add:prop1}, it follows that $h^1+\chi=-h^2+\chi$ a.s. and thus, $h^1=-h^2$ a.s., a contradiction again. Thus this completes the proof of our claim. Now for $\varepsilon>0,$ from Theorem \ref{sec2:thm1}, there are $c<v_1,v_2<1$ and $K>0$ independent of $N$ such that for all $N\geq 1,$
\begin{align}
\label{proof:eq4}
\e\left<I(|R(\vsi,\vtau)|\leq \sqrt{v_1v_2},|R(\vsi,\vtau)-u_f|\geq \varepsilon)\right>\leq K\exp\left(-\frac{N}{K}\right).
\end{align}

Recall $S_N:=\{i/N:-N\leq i\leq N\}.$ An advantage brought by the assumptions on the parameters for the two models is that under this setting it is slightly easier to find parameters to control Guerra's bound that yields the following statement: There are constants $K_1,K_2>0$ depending only on $\xi$ such that if $0<c'<c''<1$ with $\xi'(c'')-\xi'(c')<K_1$ and $(k,\mathbf{m},\mathbf{q})$ is any triplet with $q_s\leq c'$ and $m_s\geq \delta$ for some $1\leq s\leq k+1,$ then
\begin{align}
\label{proof:eq2}
p_{N,u}\leq 2\mathcal{P}_k(\mathbf{m},\mathbf{q})-\delta K_2\int_{c'}^{c''}\e F_{u}(h^1,h^2,\xi'(q))\xi''(q)dq
\end{align}  
for all $u\in S_N$ with $c''\leq |u|\leq 1,$ where 
\begin{align*}
F_u(x_1,x_2,w)&:=\left\{
\begin{array}{ll}
\e\left(\tanh(x_1+z\sqrt{w})-\tanh(x_2+z\sqrt{w})\right)^2,&\mbox{if $u>0$},\\
\e\left(\tanh(x_1+z\sqrt{w})+\tanh(x_2-z\sqrt{w})\right)^2,&\mbox{if $u<0$}
\end{array}\right.
\end{align*}
for some standard Gaussian r.v. $z.$ The proof of $(\ref{proof:eq2})$ is based on a series of applications of the Gaussian interpolation technique to the iteration scheme of the Parisi functional. One may refer to Proposition 11 \cite{Chen11} to a detailed discussion. Let us emphasize that although Proposition 11 \cite{Chen11} considers the case $h^1=h^2,$ $(\ref{proof:eq2})$ is indeed also true for identically distributed $h^1,h^2$ (see $(6.17)$ in \cite{Chen11}). Now, we let $c''=\sqrt{v_1v_2}$ and pick $c'$ with $c<c'<c''$ such that $\mu_P$ is continuous at $c'$ and $\xi'(c'')-\xi'(c')<K_1.$ By the definition of the Parisi measure $\mu_P$, it is the weak limit of a sequence of probability measures $\mu_n\in\mbox{MIN}(\varepsilon_n)$ with $\varepsilon_n\downarrow 0.$ Using this sequence and $(\ref{proof:eq2}),$ we have for all $N\geq 1,$
\begin{align*}
p_{N,u}\leq 2\mathcal{P}(\xi,h,\mu_n)-\mu_n([0,c'])K_2\int_{c'}^{c''}\e F_{u}(h^1,h^2,\xi'(q))\xi''(q)dq
\end{align*}
for all $u\in S_N$ with $c''\leq |u|\leq 1$ and letting $n$ tend to infinity implies
\begin{align*}
p_{N,u}\leq 2\mathcal{P}(\xi,h)-\mu_P([0,c'])K_2\int_{c'}^{c''}\e F_{u}(h^1,h^2,\xi'(q))\xi''(q)dq.
\end{align*}
Since $\e(h^1\pm h^2)^2\neq 0$, $\tanh$ is strictly increasing, and $c$ is the smallest number in the support of $\mu_P,$ there is a constant $\varepsilon^*>0$ independent of $N$ such that for all $N\geq 1,$
\begin{align*}
p_{N,u}\leq 2\mathcal{P}(\xi,h)-\varepsilon^*
\end{align*}
for all $u\in S_N$ with $c''\leq |u|\leq 1.$ Using this and concentration inequalities for the Gaussian r.v. $\mathcal{G}^1$, $\mathcal{G}^2$ and the sub-Gaussian r.v. $h^1$, $h^2$ as we have argued in the proof of Theorem $\ref{sec2:thm1}$, we have that
\begin{align*}
\e\left<I(|R(\vsi,\vtau)|\geq \sqrt{v_1v_2})\right>\leq K'\exp\left(-\frac{N}{K'}\right)
\end{align*}
for all $N\geq 1,$ where $K'>0$ is some constant independent of $N.$ Combining this inequality with $(\ref{proof:eq4})$ clearly completes the proof.
\end{Proof of theorem}

\bigskip

\noindent{\Large\bf Appendix}
\smallskip
\smallskip

\noindent Recall that we say a r.v. $Y$ is sub-Gaussian if there exist $d_1,d_2>0$ such that $\e \exp tY\leq d_1\exp d_2t^2$ for all $t\in\mathbb{R}.$
In this appendix, we will prove a concentration inequality for correlated sub-Gaussian r.v. that are mainly used in the proofs of Theorems \ref{thm:efc} and \ref{sec2:thm1}. 

\begin{lemma}\label{appendix:lem1}
If $Y$ is a centered sub-Gaussian r.v., then there exists some $d>0$ such that $\e\exp tY\leq \exp dt^2$ for all $t\in\mathbb{R}.$
\end{lemma}

\begin{proof}
Let $\delta>0$ be fixed and sufficiently small. If $d_1\leq \exp Md_2\delta^2$ for some $M>0,$ then 
\begin{align}
\label{appendix:lem1:proof:eq1}
d_1\exp d_2t^2\leq \exp d_2(M+1)t^2,\,\,\forall |t|\geq \delta.
\end{align}
On the other hand, since $\e Y=0,$
\begin{align*}
\e \exp tY&=1+t\e Y+\frac{t^2}{2}\e Y^2+o(t^2)\leq \exp\frac{t^2}{2}(\e Y^2+1),\,\,\forall |t|\leq \delta.
\end{align*}
This and \eqref{appendix:lem1:proof:eq1} give our assertion by letting $d:=\max(d_2(M+1),(\e Y^2+1)/2)$.
\end{proof}

\begin{proposition}[sub-Gaussian concentration inequality]\label{appendix:prop1} Let $Y^1$ and $Y^2$ be two sub-Gaussian r.v. $($not necessarily independent$)$. Let $N\geq 1.$ Suppose that $F$ is a real-valued function on $\mathbb{R}^{2N}$ satisfying the following property: For $1\leq i\leq N,$ 
\begin{align}
\begin{split}\label{appendix:prop1:eq1}
&|F(\mathbf{y}^1,\mathbf{y}^2)-F(\mathbf{y}_i^1(z^1),\mathbf{y}_i^2(z^2))|\leq \frac{1}{N}(|y_i^1|+|y_i^2|+|z^1|+|z^2|)
\end{split}
\end{align}
for all $\mathbf{y}^1=(y_1^1,\ldots,y_N^1)$, $\mathbf{y}^2=(y_1^2,\ldots,y_N^2)\in\mathbb{R}^N,$ and $z^1,z^2\in\mathbb{R},$ where $\mathbf{y}_i^j(z^j):=(y_1^j,\ldots,y_{i-1}^j,z^j,y_{i+1}^j,\ldots,y_N^j).$ Let $(Y_1^1,Y_{1}^2),\ldots,(Y_{N}^1,Y_N^2)$ be i.i.d. copies of $({Y}^1,{Y}^2)$. Set $\mathbf{Y}=(Y_1^1,\ldots,Y_N^1,Y_1^2,\ldots,Y_N^2)$. Then there exists a constant $K>0$ depending only on $Y^1,Y^2$ such that for any $\varepsilon>0,$ 
\begin{align}\label{appendix:lem1:eq1}
\p\biggl(|F(\mathbf{Y})-\e F(\mathbf{Y})|\geq \varepsilon\biggr)\leq K\exp\biggl(-\frac{N\varepsilon^2}{K}\biggr).
\end{align}
\end{proposition}

\begin{proof} 
The main idea of the proof is to use the martingale difference sequence.
Define $Y_c^1=|Y^1|-\e|Y^1|$ and $Y_{c}^2=|Y^2|-\e |Y^2|.$ Easy to see that these two r.v. are centered sub-Gaussian. From Lemma \ref{appendix:lem1}, there exists some $d>0$ such that $\e\exp tY_c^1,\e\exp tY_c^2\leq \exp dt^2$ for all $t\in\mathbb{R}.$
Let $\mathcal{F}_0$ be the trivial $\sigma$-field and $\mathcal{F}_i$ be the $\sigma$-field generated by $(Y_1^1,Y_1^2),\ldots,(Y_i^1,Y_i^2).$ Set $v_i=\e[F(\mathbf{Y})|\mathcal{F}_i]-\e[F(\mathbf{Y})|\mathcal{F}_{i-1}]$. From \eqref{appendix:prop1:eq1},
\begin{align*}
v_i&\leq \frac{1}{N}(\e|Y^{1}|+\e|Y^2|)+\frac{1}{N}(|Y_{i}^1|+|Y_{i}^2|)\\
&=\frac{2}{N}(\e|Y^{1}|+\e|Y^2|)+\frac{1}{N}\bigl(|Y_{i}^1|-\e|Y^1|+|Y_{i}^2|-\e|Y^2|\bigr).
\end{align*}
Using this and the Cauchy-Schwartz inequality give that for all $\lambda>0,$
\begin{align*}
\e \exp \lambda v_i&\leq K_N\e\exp \frac{\lambda}{N}(Y_c^1+Y_c^2)\leq K_N\exp \frac{4d\lambda^2}{N^2},
\end{align*}
where $K_N:=\exp 2N^{-1}(\e|Y^{1}|+\e|Y^2|).$ Using the exponential Chebyshev inequality and the independence of $(Y_{i}^1,Y_i^2)$'s imply
\begin{align*}
\p(F(\mathbf{Y})-\e F(\mathbf{Y})\geq \varepsilon)
&=\p\biggl(\sum_{i\leq N}v_i\geq \varepsilon\biggr)\\
&\leq \inf_{\lambda>0}\biggl(\exp(-\lambda \varepsilon)\prod_{i\leq N}\e\exp \lambda v_i\biggr)\\
&\leq K_N^N\inf_{\lambda>0} \exp\biggl(-\lambda \varepsilon+\frac{4d\lambda^2}{N}\biggr).
\end{align*}
Optimizing the last term and noting $K_N^N=\exp 2(\e|Y^{1}|+\e|Y^2|)$ yield
\begin{align}\label{appendix:prop:proof:eq1}
\p(F(\mathbf{Y})-\e F(\mathbf{Y})\geq \varepsilon)&\leq \exp 2(\e|Y^{1}|+\e|Y^2|)\exp\biggl(-\frac{\varepsilon^2N}{16d}\biggr).
\end{align}
Since the inequality \eqref{appendix:prop1:eq1} also holds for $-F$, applying \eqref{appendix:prop:proof:eq1} to $-F$, one can also estimate $\p(F(\mathbf{Y})-\e F(\mathbf{Y})\leq -\varepsilon)$ with the same upper bound as \eqref{appendix:prop:proof:eq1} and this completes our proof.
\end{proof}

\end{document}

%% file: invlat.tex
\font\tencmmib=cmmib10 \skewchar\tencmmib '60
\newfam\cmmibfam
\textfont\cmmibfam=\tencmmib

\def\lessim{\ \lower4pt\hbox{$
\buildrel{\displaystyle <}\over\sim$}\ }
\def\gessim{\ \lower4pt\hbox{$\buildrel{\displaystyle >}
\over\sim$}\ }

\def\go0{\to 0}

\def\leftitem#1{\item{\hbox to\parindent{\enspace#1\hfill}}}

\def\sg{\sigma}

\def\sg2{\sigma^2}

\def\__{_{\infty}}